\documentclass[12pt]{amsart}
\usepackage{amscd,amssymb,amsthm,amsmath,amssymb,textcomp,supertabular,rotating,longtable,enumerate,rotating,mathrsfs,mathtools,multirow}
\usepackage[matrix,arrow,curve]{xy}
\usepackage{pdflscape}

\usepackage[shortlabels]{enumitem}

\newtheorem{theorem}[equation]{Theorem}
\newtheorem*{theorem*}{Theorem}

\newtheorem{proposition}[equation]{Proposition}
\newtheorem*{proposition*}{Proposition}
\newtheorem{lemma}[equation]{Lemma}
\newtheorem{corollary}[equation]{Corollary}
\newtheorem*{corollary*}{Corollary}

\newtheorem*{problem*}{Problem}

\newtheorem*{question*}{Question}

\newtheorem*{construction*}{Construction}
\newtheorem*{maintheorem*}{Main Theorem}

\theoremstyle{definition}
\newtheorem{example}[equation]{Example}
\newtheorem*{example*}{Example}

\newtheorem*{definition*}{Definition}

\theoremstyle{remark}
\newtheorem{remark}[equation]{Remark}
\newtheorem*{remark*}{Remark}

\makeatletter\@addtoreset{equation}{section} \makeatother

\newcommand{\mumu}{\boldsymbol{\mu}}

\author{Ivan Cheltsov and Alan Thompson}

\title{K-moduli of Fano threefolds in family \textnumero 3.10}
\address{\emph{Ivan Cheltsov}
\newline
\textnormal{University of Edinburgh,  Edinburgh, Scotland}
\newline
\textnormal{\texttt{i.cheltsov@ed.ac.uk}}}

\address{\emph{Alan Thompson}
\newline
\textnormal{Loughborough University, Loughborough, England}
\newline
\textnormal{\texttt{a.m.thompson@lboro.ac.uk}}}

\begin{document}

\begin{abstract}
We find all K-polystable limits of smooth Fano threefolds in family~\textnumero 3.10.
\end{abstract}

\maketitle
\addtocontents{toc}{\protect\setcounter{tocdepth}{1}}
\tableofcontents

Throughout this paper, all varieties are assumed to be projective and defined over~$\mathbb{C}$.

\section{Introduction}
\label{section:intro}

Let $Q$ be a~smooth quadric threefold in $\mathbb{P}^4$, let $C_1$ and $C_2$ be disjoint smooth irreducible conics in the~quadric $Q$.
Then we can choose coordinates $x$, $y$, $z$, $t$, $w$ on~$\mathbb{P}^4$ such that
\begin{align*}
C_1&=\{x=0,y=0,w^2+zt=0\},\\
C_2&=\{z=0,t=0,w^2+xy=0\},
\end{align*}
and $Q$ is one of the~following quadrics:
\begin{align*}
Q&=\big\{w^2+xy+zt=a(xt+yz)+b(xz+yt)\big\},\ \text{$(a,b)\in\mathbb{C}^2$ such that $a\pm b\ne \pm 1$},\tag{$\beth$}\\
Q&=\big\{w^2+xy+zt=a(xt+yz)+xz\big\},\ \text{$a\in\mathbb{C}$ such that $a\ne\pm 1$},\tag{$\gimel$}\\
Q&=\big\{w^2+xy+zt=xt+xz\big\}.\tag{$\daleth$}
\end{align*}

Let $\pi\colon X\to Q$ be the~blow up of the~quadric threefold $Q$ along the~conics $C_1$~and~$C_2$.
Then~$X$ is a~smooth Fano~threefold in the~deformation family \textnumero 3.10,
and all smooth  members of this deformation family  can be obtained in this way.

Alternatively, we can describe $X$ as
a~complete intersection in $\mathbb{P}^1\times\mathbb{P}^1\times\mathbb{P}^4$ of three smooth divisors of degree $(1,0,1)$, $(0,1,1)$, $(0,0,2)$.
Namely, if $Q$ is the~quadric $(\beth)$,~then
\begin{equation}
\label{equation:beth}\tag{$\beth$}
X=\big\{u_1x=v_1y,u_2z=v_2t,w^2+xy+zt=a(xt+yz)+b(xz+yt)\big\}
\end{equation}
where $([u_1:v_1],[u_2:v_2])$ are coordinates on $\mathbb{P}^1\times\mathbb{P}^1$.
If $Q$ is the~quadric $(\gimel)$, then
\begin{equation}
\label{equation:gimel}\tag{$\gimel$}
X=\big\{u_1x=v_1y,u_2z=v_2t,w^2+xy+zt=a(xt+yz)+xz\big\}
\end{equation}
Finally, if $Q$ is the~quadric $(\daleth)$, then
\begin{equation}
\label{equation:daleth}\tag{$\daleth$}
X=\big\{u_1x=v_1y,u_2z=v_2t,w^2+xy+zt=xt+xz\big\}
\end{equation}

We know all K-polystable smooth Fano threefolds in the~deformation family \textnumero 3.10.
Namely, it has been proved in \cite[\S 5.17]{Book} that $X$ is K-polystable $\iff$ $Q$ is the~quadric~$(\beth)$.

The goal of this paper is prove the~following result:

\begin{maintheorem*}
K-polystable limits of smooth Fano threefolds in the~family \textnumero 3.10 are
complete intersections in $\mathbb{P}^1\times\mathbb{P}^1\times\mathbb{P}^4$ which can be described as follows:
\begin{equation}
\label{equation:K-beth}\tag{$\beth$}
\big\{u_1x=v_1y,u_2z=v_2t,w^2+xy+zt=a(xt+yz)+b(xz+yt)\big\},
\end{equation}
\begin{equation}
\label{equation:V-aleph}\tag{$\aleph$}
\big\{u_1x=v_1y,u_2z=v_2t,w^2+r(x+y)^2+r(z+t)^2=(2s+2)(xt+yz)+(2s-2)(xz+yt)\big\},
\end{equation}
where $(a,b)\in\mathbb{C}^2$ and $[r:s]\in\mathbb{P}^1$.
\end{maintheorem*}

K-polystable Fano threefolds in the Main Theorem form an irreducible two-dimensional component of the K-moduli of smoothable Fano threefolds.
In Section~\ref{section:K-moduli}, 
we show that this component is a~connected component of this moduli space,
and it is isomorphic to~$\mathbb{P}^2$.
Another two-dimensional component of this moduli~space has been described in \cite{Gokova},
and its one-dimensional components are described~in~\cite{London}.

Let us say few words about the threefolds in the Main Theorem.
Let $\mathscr{X}$ be one of them.
Then, using the~natural projection $\mathscr{X}\to\mathbb{P}^4$, we get a~birational morphism $\varpi\colon\mathscr{X}\to\mathscr{Q}$,
where $\mathscr{Q}$ is a~(possibly singular) irreducible quadric in $\mathbb{P}^4$ such that either
\begin{equation}
\label{equation:K-beth-quadric}\tag{$\beth$}
\mathscr{Q}=\big\{w^2+xy+zt=a(xt+yz)+b(xz+yt)\big\}
\end{equation}
or
\begin{equation}
\label{equation:V-aleph-quadric}\tag{$\aleph$}
\mathscr{Q}=\big\{w^2+r(x+y)^2+r(z+t)^2=(2s+2)(xt+yz)+(2s-2)(xz+yt)\big\},
\end{equation}
where $(a,b)\in\mathbb{C}^2$ and $[r:s]\in\mathbb{P}^1$.
The morphism $\varpi$ is a~blow up of two conics:
\begin{align*}
\mathscr{C}_1&=\mathscr{Q}\cap \{x=0,y=0\},\\
\mathscr{C}_2&=\mathscr{Q}\cap\{z=0,t=0\}.
\end{align*}
These conics are contained in the~smooth locus of the~quadric $\mathscr{Q}$, and
\begin{itemize}
\item $\mathscr{C}_1$ and~$\mathscr{C}_2$ are smooth if $\mathscr{Q}$ is the~quadric ($\beth$),
\item $\mathscr{C}_1$ and~$\mathscr{C}_2$ are reduced and singular if $\mathscr{Q}$ is the~quadric ($\aleph$) with $[r:s]\ne [0:1]$,
\item $\mathscr{C}_1$ and~$\mathscr{C}_2$ are non-reduced if $\mathscr{Q}$ is the~quadric ($\aleph$) with  $[r:s]=[0:1]$.
\end{itemize}
If $\mathscr{Q}$ is the~quadric ($\beth$), then $\mathscr{X}$ is singular $\iff$ $\mathscr{Q}$ is singular $\iff$ $a\pm b\pm 1=0$.
In~this case, the~singularities of the~quadric $\mathscr{Q}$ can be described as follows:
\begin{itemize}
\item if $a\pm b\pm 1=0$ and $ab\ne 0$, then $\mathscr{Q}$ has one singular point,
\item if $a\pm b\pm 1=0$ and $ab=0$, then $\mathscr{Q}$ is singular along a~line.
\end{itemize}
Similarly, if $\mathscr{Q}$ is the~quadric~($\aleph$), then $\mathscr{X}$ is singular,
because both $\mathscr{C}_1$ and~$\mathscr{C}_2$ are singular.
In~this case, $\mathscr{Q}$ is smooth $\iff$ $[r:s]\ne[\pm 1:1]$,
and $|\mathrm{Sing}(\mathscr{Q})|=1$ when  $[r:s]=[\pm 1:1]$.

\medskip
\noindent
\textbf{Acknowledgements.}
We started this research project in New York back in October~2022 during the~conference \emph{MMP and Moduli},
and we completed it in Banff in February 2023 during the~workshop \emph{Explicit Moduli Problems in Higher Dimensions}.
We would like~to thank the~Simons Foundation and Banff International Research Station for very good working conditions.
We would like to thank Hamid Abban and Kento Fujita for comments,
and Yuchen Liu for his help with arguments in Section~\ref{section:K-moduli}.

Ivan Cheltsov is supported by EPSRC grant \textnumero EP/V054597/1.

Alan Thompson is supported by EPSRC grant \textnumero EP/V005545/1.

\section{First GIT quotient (bad moduli space)}
\label{section:GIT-1}

In this section we describe a~compact moduli space for Fano threefolds in family \textnumero 3.10,
by applying GIT techniques to the~description of such threefolds as blow-ups of quadrics along two smooth disjoint conics.
As in Section~\ref{section:intro}, we fix two smooth disjoint conics
\begin{align*}
C_1&=\{x=0,y=0,w^2+zt=0\}\subset\mathbb{P}^4,\\
C_2&=\{z=0,t=0,w^2+xy=0\}\subset\mathbb{P}^4,
\end{align*}
where $[x:y:z:t:w]\in\mathbb{P}^4$.
Let $Q$ be a~quadric threefold in $\mathbb{P}^4$ that contains $C_1$ and $C_2$.
Then $Q$ is given by an equation of the~form
\begin{equation}
\label{equation:quadric}
\alpha(xy+zt+w^2)+\beta xz+\gamma xt+\delta yz+\epsilon yt=0,
\end{equation}
where $[\alpha:\beta:\gamma:\delta:\epsilon] \in \mathbb{P}^4$.
Let $\pi\colon X\to Q$ be the~blow up of of the~conics $C_1$~and~$C_2$.
If $Q$ is smooth, then~$X$ is a~smooth Fano~threefold in the~deformation family \textnumero 3.10.

Note that equations of the~form \eqref{equation:quadric} are preserved by the~action of $(\mathbb{C}^\ast)^2$ given by
$$
[x:y:z:t:w]\mapsto\Big[\lambda x: \tfrac{1}{\lambda}y: \mu z: \tfrac{1}{\mu} t: w\Big]
$$
for $(\lambda,\mu)\in(\mathbb{C}^\ast)^2$. This induces a~$(\mathbb{C}^\ast)^2$-action on the~parameter space $\mathbb{P}^4$ given by
$$
[\alpha:\beta:\gamma:\delta:\epsilon] \mapsto \Big[\alpha:\lambda\mu \beta: \tfrac{\lambda}{\mu} \gamma: \tfrac{\mu}{\lambda} \delta: \tfrac{1}{\lambda\mu} \epsilon\Big]
$$
for $(\lambda,\mu)\in(\mathbb{C}^\ast)^2$.
The GIT quotient by this action is isomorphic to $\mathbb{P}^2$
with coordinates
$$
[\mathbf{x}_0:\mathbf{x}_1:\mathbf{x}_2] =[\alpha^2: \beta\epsilon: \gamma\delta].
$$
We now classify the~unstable, semistable, and stable points, and corresponding quadrics.

\subsection{Unstable points.}
The unstable points $[\alpha:\beta:\gamma:\delta:\epsilon] \in \mathbb{P}^4$ and the~corresponding quadric threefolds can be described as follows:
\begin{align*}
\alpha = \beta =\gamma = 0&\ \text{and}\ Q=\{y(\delta z+\epsilon t)=0\},\\
\alpha = \beta = \delta = 0&\ \text{and}\ Q=\{t(\gamma x+\epsilon y)=0\}, \\
\alpha = \gamma = \epsilon = 0&\ \text{and}\ Q=\{z(\beta x+\delta y)=0\}, \\
\alpha = \delta = \epsilon = 0&\ \text{and}\ Q=\{x( \beta z+\gamma t)=0\}.
\end{align*}
In these cases, the~quadric $Q$ degenerates to a~union of two hyperplanes.

\subsection{Stable points.}
A point $[\alpha:\beta:\gamma:\delta:\epsilon] \in \mathbb{P}^4$ is stable if and only if its orbit is closed and its stabiliser is finite.
This occurs if and only if all of $\beta$, $\gamma$, $\delta$ and $\epsilon$~are~nonzero,
\mbox{corresponding} to points in the~GIT quotient $\mathbb{P}^2$  away from the~lines $\{\mathbf{x}_1=0\}$ and $\{\mathbf{x}_2=0\}$.
We separate the~corresponding quadric threefolds into three cases.
\begin{enumerate}[(S1)]
\item Stable orbits with $\alpha \neq 0$ correspond to points in $\mathbb{P}^2$ with all of $\mathbf{x}_0, \mathbf{x}_1, \mathbf{x}_2$ nonzero.
Taking the~affine chart $\mathbf{x}_0 = 1$, each such orbit contains quadrics of the~form
$$
w^2+xy+zt+a(xt+yz)+b(xz+yt)=0,
$$
for $(a,b)\in(\mathbb{C}^*)^2$ given by $a = \pm \sqrt{\mathbf{x}_2}$ and $b = \pm \sqrt{\mathbf{x}_1}$.
The~quadric is singular if and only if $a \pm b \pm 1 = 0$,
which corresponds to the~affine curve
$$
1-2\mathbf{x}_1-2\mathbf{x}_2+\mathbf{x}_1^2-2\mathbf{x}_1\mathbf{x}_2+\mathbf{x}_2^2=0.
$$

\item Orbits with $\alpha=0$ and $\gamma \delta \neq \beta \epsilon$ correspond to points in the~quotient $\mathbb{P}^2$
that satisfy the~following conditions: $\mathbf{x}_0 = 0$ and $\mathbf{x}_1 \neq \mathbf{x}_2$. So, taking the~affine chart $\mathbf{x}_1=1$,
we see that each such orbit contains a~quadric of the~form
$$
a(xt+yz)+xz+yt= 0,
$$
for $a \in \mathbb{C}^* \setminus \{\pm 1\}$ given by $a = \pm\sqrt{\mathbf{x}_2}$.
Such quadrics have one singular point.

\item There is a~unique stable orbit with $\alpha=0$ and $\gamma \delta = \beta \epsilon \neq 0$, corresponding to the~point $[\mathbf{x}_0:\mathbf{x}_1:\mathbf{x}_2] = [0:1:1]$ in the~quotient. It contains the~threefold
$$
(x+y)(z+t)=0,
$$
which is a~union of two hyperplanes.
\end{enumerate}

\subsection{Strictly semistable points.}
A point $[\alpha:\beta:\gamma:\delta:\epsilon]$ is strictly semistable if and only if
it is not unstable, and at least one of $\beta,\gamma,\delta,\epsilon$ is zero.
Such points have non-closed orbits or infinite stabilisers.

\begin{example}
If $\beta\ne 0$, $\gamma\ne 0$ and $\delta\ne 0$, then the~orbit of $[\alpha:\beta:\gamma:\delta:0]$ contains
the~point $[\alpha:0:\gamma:\delta:0]$ in its closure, and the~orbit of $[\alpha:0:\gamma:\delta:0]$ is closed but has infinite stabiliser given by $(\lambda,\lambda)\in (\mathbb{C}^\ast)^2$.
\end{example}

The strictly semistable points belonging to minimal orbits (strictly polystable points) are described below, along with their corresponding quadric threefolds.
\begin{enumerate}[(SS1)]
\item All points of the~form $[\alpha:0:\gamma:\delta:0]$ or $[\alpha:\beta:0:0:\epsilon]$,
where $\alpha\beta\gamma\delta\epsilon\ne 0$, and either $\alpha^2 \neq \gamma\delta$ or $\alpha^2 \neq \beta\epsilon$, respectively.
Such orbits lie over points in the~quotient
\begin{align*}
\Big\{[\mathbf{x}_0:0:\mathbf{x}_2]\ \big|\ \mathbf{x}_0, \mathbf{x}_2 \in \mathbb{C}^*,\ \mathbf{x}_0 \neq \mathbf{x}_2\Big\},\\
\Big\{[\mathbf{x}_0:\mathbf{x}_1:0]\ \big|\ \mathbf{x}_0, \mathbf{x}_1 \in \mathbb{C}^*,\ \mathbf{x}_0 \neq \mathbf{x}_1\Big\}.
\end{align*}
The corresponding quadrics are
\begin{align*}
\big\{w^2+xy+zt+a(xt+yz)&=0\big\},\\
\big\{w^2+xy+zt+b(xz+yt)&=0\big\},
\end{align*}
where $a,b \in \mathbb{C}^* \setminus \{\pm 1\}$. They are smooth.

\item The orbit of the~point $[1:0:0:0:0]$.
It gives the~point $[1:0:0]$ in the~quotient.
The corresponding quadric threefold is $\{w^2+xy+zt=0\}$ --- it is smooth.

\item Two orbits of all points $[0:0:\gamma:\delta:0]$ and $[0:\beta:0:0:\epsilon]$ such that $\beta\gamma\delta\epsilon\ne 0$,
which lie over the~points $[0:0:1]$ and $[0:1:0]$ in the~quotient $\mathbb{P}^2$, respectively.
This gives two singular quadrics: $\{xt+yz=0\}$ and $\{xz+yt=0\}$.

\item Two orbits consisting of points of the~form $[\alpha:0:\gamma:\delta:0]$ or $[\alpha:\beta:0:0:\epsilon]$,
where $\alpha\beta\gamma\delta\epsilon\ne 0$, and either $\alpha^2 = \gamma\delta$ or $\alpha^2=\beta\epsilon$, respectively.
These orbits lie over the~points $[1:0:1]$ and $[1:1:0]$ in the~quotient,
and the~quadrics are
\begin{align*}
w^2+xy+zt+xt+yz&=0,\\
w^2+xy+zt+xz+yt&=0.
\end{align*}
Both of them are irreducible and singular along the~line $\{w=y-t=x-z=0\}$.
\end{enumerate}

\subsection{Quotient space}
Note that the~space of quadrics of the~form \eqref{equation:quadric}
admits a~discrete action of the~group $\mumu_2^3$ given by
\begin{align*}
\sigma_1\colon [x:y:z:t:w] &\mapsto [y:x:z:t:w],\\
\sigma_2\colon [x:y:z:t:w] &\mapsto [z:t:x:y:w],\\
\sigma_3\colon [x:y:z:t:w] &\mapsto [x:y:z:t:-w].
\end{align*}
These involutions act on the~parameter space $\mathbb{P}^4$ by
\begin{align*}
\sigma_1\colon [\alpha:\beta:\gamma:\delta:\epsilon]&\mapsto [\alpha:\delta:\epsilon:\beta:\gamma],\\
\sigma_2\colon [\alpha:\beta:\gamma:\delta:\epsilon]&\mapsto [\alpha:\beta:\delta:\gamma:\epsilon],\\
\sigma_3\colon [\alpha:\beta:\gamma:\delta:\epsilon]&\mapsto [\alpha:\beta:\gamma:\delta:\epsilon].
\end{align*}
The involutions $\sigma_2$ and $\sigma_3$ act trivially on the~GIT quotient, whilst $\sigma_1$ acts by
$$
\sigma_1\colon [\mathbf{x}_0:\mathbf{x}_1:\mathbf{x}_2] \mapsto [\mathbf{x}_0:\mathbf{x}_2:\mathbf{x}_1].
$$
Therefore, to get a~moduli space from our GIT quotient,
we must further quotient it by the~involution $\sigma_1$ to get a~copy of the~weighted projective space $\mathbb{P}(1,1,2)$ with coordinates
$$
[\xi:\eta:\zeta]=\big[\mathbf{x}_0:\mathbf{x}_1+\mathbf{x}_2:(\mathbf{x}_1 - \mathbf{x}_2)^2\big]=\big[\alpha^2:\beta\epsilon+\gamma\delta:(\beta\epsilon-\gamma\delta)^2\big].
$$
Over this quotient it is easy to describe where the~stable and semistable threefolds occur.
Namely, the~stable threefolds occur where $\eta^2 \neq \zeta$ as follows:
\begin{enumerate}[(S1)]
\item the~locus with $\xi \neq 0$ and $\eta^2 \neq \zeta$, and singular quadrics lie over $\xi^2-\xi\eta+\zeta=0$,
\item the~locus $\{[0:\eta:\zeta]\mid \eta^2 \neq \zeta \text{ and } \zeta \neq 0\}$,
\item the~point $[\xi:\eta:\zeta] = [0:1:0]$.
\end{enumerate}
Likewise, the~strictly semistable quadrics occur along the~curve $\eta^2 = \zeta$ as follows:
\begin{enumerate}[(SS1)]
\item the~locus where $\eta^2 = \zeta$ and $[\xi:\eta:\zeta] \notin \{[1:0:0],  [0:1:1], [1:1:1]\}$,
\item the~point $[\xi:\eta:\zeta] =[1:0:0]$,
\item the~point $[\xi:\eta:\zeta] = [0:1:1]$,
\item the~point $[\xi:\eta:\zeta] = [1:1:1]$.
\end{enumerate}

Unfortunately, the~constructed moduli space is not the~moduli space of K-polystable threefolds we are looking for,
because the~threefold $X$ obtained by blowing up the~quadric in the~class (S3) is reducible.
Moreover, all threefolds obtained by blowing up quadrics in the~class (S2) are isomorphic,
and the~quadrics in the~class $\mathrm{(S2)}$ give K-unstable threefolds:

\begin{lemma}
\label{lemma:K-unstable}
Let $Q$ be the quadric $\{a(xt+yz)+xz+yt=0\}\subset\mathbb{P}^4$ for $a\in\mathbb{C}\setminus\{0,\pm 1\}$,
and let $\pi\colon X\to Q$ be the~blow up of the~quadric $Q$ along $C_1$ and $C_2$.
Then $X$ is K-unstable.
\end{lemma}

\begin{proof}
Observe that both $C_1$ and $C_2$ do not contain the~singular point of the~quadric~$Q$,
the~quadric threefold $Q$ is a~cone over $\mathbb{P}^1\times\mathbb{P}^1$ with vertex at the~point $[0:0:0:0:1]$,
and there exists the~following commutative diagram:
$$
\xymatrix{
&\widetilde{X}\ar@{->}[dl]_{\varphi}\ar@{->}[dr]^{\varpi}\ar@/^2.5pc/@{->}[ddrr]^{\nu}&\\%
X\ar@{->}[dr]_{\pi}&&\widetilde{Q}\ar@{->}[dl]^{\phi}\ar@{->}[dr]^{\upsilon}&\\%
&{Q}&&\mathbb{P}^1\times\mathbb{P}^1}
$$
where $\phi$ and $\varphi$ are~blow ups of the~singular points $\mathrm{Sing}(Q)$ and $\mathrm{Sing}(X)$,
$\upsilon$ is a~$\mathbb{P}^1$-bundle,
$\varpi$ is the~blow up of the~preimages of the~conics $C_1$ and $C_2$,
and $\nu$ is a~conic bundle.

Let $E_1$ and $E_2$ be proper transforms on the~threefold $\widetilde{X}$ of the~$\pi$-exceptional surfaces such that $\pi\circ\varphi(E_1)=C_1$ and $\pi\circ\varphi(E_2)=C_2$.
Then $\upsilon\circ\varpi(E_1)$ and $\upsilon\circ\varpi(E_2)$ are~disjoint rulings of the~surface $\mathbb{P}^1\times\mathbb{P}^1$,
because planes spanned by $C_1$ and $C_2$ are contained in $Q$.
Therefore, we may assume that both $\upsilon\circ\varpi(E_1)$ and $\upsilon\circ\varpi(E_2)$ are divisors of degree $(1,0)$.
Let $H_1$ and $H_2$ be the~pull back on $\widetilde{X}$ of the~divisors on $\mathbb{P}^1\times\mathbb{P}^1$ of degree $(1,0)$ and $(0,1)$,
let $S_1$ and $S_2$ be proper transforms on $\widetilde{X}$ of the~planes in $Q$ that contain $C_1$ and $C_2$,
and let $F$ be the~$\varphi$-exceptional surface. Then
$H_1\sim S_1+E_1$ and $H_1\sim S_2+E_2$, so that
$$
\varphi^*(-K_X)\sim 3(H_1+H_2)+3F-E_1-E_2\sim_{\mathbb{Q}} \frac{3}{2}(S_1+S_2)+H_2+3F+\frac{1}{2}(E_1+E_2).
$$
Now, take $u\in\mathbb{R}_{\geqslant 0}$. Then
$$
\varphi^*(-K_X)-uF\sim_{\mathbb{R}} 3(H_1+H_2)+(3-u)F-E_1-E_2\sim_{\mathbb{R}} \frac{3}{2}(S_1+S_2)+H_2+(3-u)F+\frac{1}{2}(E_1+E_2),
$$
so this divisor is pseudoeffective $\iff$ $u\leqslant 3$.
Moreover, this divisor is nef $\iff$ $u\leqslant 1$.
Furthermore, if $u\in(1,3)$, then its Zariski decomposition can be described as follows:
$$
\varphi^*(-K_X)-uF\sim_{\mathbb{R}}\underbrace{\varphi^*(-K_X)-uF-\frac{u-1}{2}\big(S_1+S_2\big)}_{\text{positive part}}+\underbrace{\frac{u-1}{2}\big(S_1+S_2\big)}_{\text{negative part}}.
$$
Thus, we see that
\begin{multline*}
\quad\quad\quad\beta(F)=2-\frac{1}{26}\int\limits_{0}^{1}\big(3(H_1+H_2)+(3-u)F-E_1-E_2\big)^3du-\\
-\frac{1}{26}\int\limits_{2}^{3}\Big(3(H_1+H_2)+(3-u)F-E_1-E_2-\frac{u-1}{2}\big(2H_1-E_1-E_2\big)\Big)^3du.
\end{multline*}
To compute these integrals, observe that
\begin{align*}
H_1^2&=0, & H_2^2&=0, & E_1\cdot E_2&=0, & E_1\cdot F&=0, & E_2\cdot F&=0,\\
F^3&=2, & H_1\cdot F^2&=-1, & H_2\cdot F^2&=-1, & E_1^3&=-4, & E_2^3&=-4, \\
F\cdot H_1\cdot H_2&=1, & H_1\cdot E_1&=0, & H_1\cdot E_2&=0, & H_2\cdot E_1^2&=-2, & H_2\cdot E_2^2&=-2.
\end{align*}
This gives $\beta(F)=-\frac{3}{52}<0$, which implies that $X$ is K-unstable \cite{Fujita2019,Li}.
\end{proof}

\section{Second GIT quotient (good moduli space)}
\label{section:GIT-2}

In this section, we construct another compact GIT moduli space for Fano threefolds in the~deformation family~\textnumero~3.10.
Recall that all K-polystable smooth Fano threefolds in this family are contained in class $(\beth)$,
and the~quadric $Q$ in this case is invariant under the~action of $\mumu_2^3$ generated by the~following three involutions:
\begin{align*}
\tau_1\colon [x:y:z:t:w] &\mapsto [y:x:t:z:w],\\
\tau_2\colon [x:y:z:t:w] &\mapsto [z:t:x:y:w],\\
\tau_3\colon [x:y:z:t:w] &\mapsto [x:y:z:t:-w].
\end{align*}
In light of this, we take $Q$ to be a~quadric invariant under this $\mumu_2^3$-action:
\begin{equation}
\label{equation:newquadric}
\alpha w^2 + \tfrac{\beta}{2}(x^2+y^2+z^2+t^2) + \gamma(xt+yz) + \delta(xz+yt) + \epsilon(xy+zt) = 0,
\end{equation}
where $[\alpha:\beta:\gamma:\delta:\epsilon] \in\mathbb{P}^4$. As before, we take $C_1$ and $C_2$ to be the~two conics given by the~intersection with the~ planes $\{x=y=0\}$ and $\{z=t=0\}$, namely
\begin{align*}
C_1 &= \{x=0,\, y=0,\, \alpha w^2 + \tfrac{\beta}{2}(z^2+t^2) + \epsilon zt = 0\},\\
C_2 &= \{z=0,\, t=0,\, \alpha w^2 + \tfrac{\beta}{2}(x^2+y^2) + \epsilon xy = 0\}.
\end{align*}
Unlike the~construction from Section \ref{section:GIT-1}, this general form allows $C_1$ and $C_2$ to degenerate.

A simple computer calculation shows that equations of the~form \eqref{equation:newquadric} are preserved by a~subgroup of $\mathrm{PGL}(5,\mathbb{C})$ isomorphic to $(\mathbb{C}^\ast)^2$, given explicitly by matrices of the~form
$$\begin{pmatrix}
\lambda + \mu & \lambda - \mu & 0 & 0 & 0 \\
\lambda - \mu & \lambda + \mu & 0 & 0 & 0 \\
0 & 0 & \lambda - \mu & \lambda + \mu & 0 \\
0 & 0 & \lambda + \mu & \lambda - \mu & 0 \\
0 & 0 & 0 & 0 & -\tfrac{2}{\lambda^2\mu^2}
\end{pmatrix}$$
acting on the~coordinates $[x:y:z:t:w]$, where $(\lambda,\mu) \in (\mathbb{C}^\ast)^2$ and we have chosen our matrix representative to have determinant $2^5 = 32$ (to avoid having lots of factors of $\frac{1}{2}$). This induces a~$(\mathbb{C}^\ast)^2$-action on the~parameter space $\mathbb{P}^4$ given by
$$\begin{pmatrix}
\tfrac{4}{\lambda^4\mu^4} & 0 & 0 & 0 \\
0 & 2(\lambda^2+ \mu^2) & 0 & 0 & 2(\lambda^2-\mu^2)\\
0 & 0 & 2(\lambda^2-\mu^2) & 2(\lambda^2 + \mu^2) & 0 \\
0 & 0 & 2(\lambda^2+\mu^2) & 2(\lambda^2 - \mu^2) & 0 \\
0 & 2(\lambda^2- \mu^2) & 0 & 0 & 2(\lambda^2+\mu^2)
\end{pmatrix}
$$
acting on the~coordinates $[\alpha:\beta:\gamma:\delta:\epsilon]$, where $(\lambda,\mu) \in (\mathbb{C}^\ast)^2$.

The ring of invariants for this action is generated by polynomials of the~form
$$
\alpha f_1f_2g_1g_2
$$
with $f_i \in \{(\beta+\epsilon),\,(\gamma+\delta)\}$ and $g_i \in \{(\beta-\epsilon),\,(\gamma-\delta)\}$.
There are nine such \mbox{polynomials}, corresponding to coordinates on $\mathbb{P}^8$.
One can show that the~GIT quotient of the~parameter space is isomorphic to $\mathbb{P}^1 \times \mathbb{P}^1 \subset \mathbb{P}^8$, embedded by the~Segre embedding of bidegree $(2,2)$.
The~isomorphism with $\mathbb{P}^1 \times \mathbb{P}^1$ is given explicitly by
$$
\big([\mathbf{x}_0:\mathbf{x}_1],[\mathbf{y}_0:\mathbf{y}_1]\big) = \big([\beta+\epsilon:\gamma + \delta], [\beta - \epsilon:\gamma - \delta]\big).
$$

We now classify the~unstable and stable points, and corresponding quadrics. There are no strictly semistable points under this action.

\subsection{Unstable points.}
The unstable points $[\alpha:\beta:\gamma:\delta:\epsilon] \in \mathbb{P}^4$ come in two types. The first are the~points with $\alpha = 0$, which have corresponding quadrics
$$
Q=\Big\{\tfrac{\beta}{2}(x^2+y^2+z^2+t^2) + \gamma(xt+yz) + \delta(xz+yt) + \epsilon(xy+tz) = 0\Big\}.
$$
In this case, the~quadric $Q$ degenerates to a~cone over a~quadric surface. The second type are those with either
\begin{align*}
\beta - \epsilon = \gamma -  \delta = 0&\ \text{and}\ Q=\Big\{\alpha w^2 + \tfrac{\beta}{2}((x+y)^2+(z+t)^2) + \gamma(x+y)(z+t)=0\Big\}, \\
\beta + \epsilon = \gamma +  \delta = 0&\ \text{and}\ Q=\Big\{\alpha w^2 + \tfrac{\beta}{2}((x-y)^2+(z-t)^2) + \gamma(x-y)(z-t) =0\Big\}.
\end{align*}
In these cases $Q$ is singular along a~line and both of the~conics $C_1$ and $C_2$ become reducible.

\subsection{Stable points.}
The stable points are those that are not unstable --- there are no strictly semistable points.
We separate the~ corresponding quadric threefolds into 2 cases.

\begin{enumerate}[(S1)]
\item (Class $(\beth)$) Stable orbits with $\alpha \neq 0$ and $\beta \neq \pm \epsilon$ correspond to points with
$$
\mathbf{x}_0\mathbf{y}_0 \neq 0
$$
in the~quotient space $\mathbb{P}^1 \times \mathbb{P}^1$.
Taking the~affine chart with $\mathbf{x}_0=1$ and $\mathbf{y}_0=-1$, we see that each such orbit contains quadrics of the~form
$$
w^2+xy+zt+a(xt+yz)+b(xz+yt)=0,
$$
for $a,b\in \mathbb{C}$ given by $a = \frac{1}{2}(\mathbf{x}_1 + \mathbf{y}_1)$ and $b = \frac{1}{2}(\mathbf{x}_1 - \mathbf{y}_1)$.
This~quadric is singular if and only if $a \pm b \pm 1 = 0$,
corresponding to $\mathbf{x}_1 = \pm 1$ or $\mathbf{y}_1 = \pm 1$.

\item (Class $(\aleph)$) Stable orbits with $\alpha\neq 0$ and $\beta = -\epsilon$ or $\beta = \epsilon$,
which correspond to points in the~quotient $\mathbb{P}^1 \times \mathbb{P}^1$ with $\mathbf{x}_0 = 0$ or $\mathbf{y}_0 = 0$, respectively.
Taking~the~affine chart $\mathbf{x}_1=1$ or $\mathbf{y}_1=1$, respectively,
we see that such orbits contain the~quadric
$$
w^2 + r(x-y)^2+r(z-t)^2 + (2s+2)(xt+yz)+(2-2s)(xz+yt)= 0
$$
or
$$
w^2 + r(x+y)^2+r(z+t)^2 + (2s+2)(xt+yz)+(2s-2)(xz+yt)= 0
$$
for $[r:s] \in \mathbb{P}^1$ given by $[r:s] = [2\mathbf{y}_0:2\mathbf{y}_1]$ or $[r:s] = [2\mathbf{x}_0:2\mathbf{x}_1]$, respectively.
Generically this defines a~smooth quadric where the~conics $C_1$ and $C_2$ are reduced and singular (pairs of lines).
The exceptions are
\begin{itemize}
\item $[r:s] = [\pm 1 :1]$, when the~quadric acquires an isolated singularity,
\item $[r:s] = [0:1]$, when $C_1$ and $C_2$ become non-reduced (double lines).
\end{itemize}
\end{enumerate}

In Sections \ref{section:beth} and \ref{section:aleph}, we will show that all Fano threefolds $X$ obtained as blow-ups of quadrics $Q$ from classes $(\beth)$ and $(\aleph)$ above are K-polystable.

\subsection{Quotient space.}
\label{section:quotient-space}
Observe that the~space of quadric threefolds of the~form \eqref{equation:newquadric}
admits an~additional $\mumu_2$-action given by
$$
[x:y:z:t:w]\mapsto [x:-y:-z:t:w],
$$
which acts on the~parameter space $\mathbb{P}^4$ as
$$
[\alpha:\beta:\gamma:\delta:\epsilon]\mapsto [\alpha:\beta:\gamma:-\delta:-\epsilon].
$$
This involution induces an involution on the~GIT quotient given by
$$
\big([\mathbf{x}_0:\mathbf{x}_1],[\mathbf{y}_0:\mathbf{y}_1]\big) \mapsto \big([\mathbf{y}_0:\mathbf{y}_1],[\mathbf{x}_0:\mathbf{x}_1]\big).
$$
So, to get a~moduli space from our GIT quotient,
we must further quotient it by this involution. The result is a~copy of $\mathbb{P}^2$ with coordinates
$$
[\xi:\eta:\zeta]=[\mathbf{x}_0\mathbf{y}_0:\mathbf{x}_1\mathbf{y}_1:\mathbf{x}_0\mathbf{y}_1+\mathbf{x}_1\mathbf{y}_0] = [\beta^2-\epsilon^2:\gamma^2-\delta^2: 2\beta\gamma-2\epsilon\delta].
$$
Over this quotient it is easy to describe where each of the~types of stable threefolds occur:
\begin{enumerate}[(S1)]
\item the~locus with $\xi \neq 0$, singular quadrics lie over the~lines $\xi+\eta = \pm \zeta$.
\item the~locus with $\xi=0$, singular quadrics lie over the~points $[0:1:1]$ and $[0:-1:1]$,
and the~conics $C_1$ and $C_2$ become non-reduced at the~point $[0:1:0]$.
\end{enumerate}
Note that this quotient identifies the two normal forms in class $(\aleph)$.

\section{Threefolds in the~class $(\beth)$}
\label{section:beth}

Let $X$ be the~complete intersection
$$
\big\{u_1x=v_1y, u_2z=v_2t, w^2+xy+zt=a(xt+yz)+b(xz+yt)\big\}\subset\mathbb{P}^1\times\mathbb{P}^1\times\mathbb{P}^4,
$$
where $(a,b)\in\mathbb{C}^2$, and $([u_1:v_1],[u_2:v_2],[x:y:z:t:w])$ are coordinates on $\mathbb{P}^1\times\mathbb{P}^1\times\mathbb{P}^4$.
Then $X$ is singular $\iff$ $a\pm b\pm 1=0$. If $a\pm b\pm 1=0$ and $ab\ne 0$, then
$$
\mathrm{Sing}(X)=\left\{\aligned
&\big([1:1],[1:1],[1:1:-1:-1:0]\big)\ \text{if $a+b+1=0$}, \\
&\big([1:1],[1:1],[1:1:1:1:0]\big)\ \text{if $a+b-1=0$}, \\
&\big([-1:1],[-1:1],[1:-1:1:-1:0]\big)\ \text{if $a-b+1=0$}, \\
&\big([-1:1],[-1:1],[1:-1:-1:1:0]\big)\ \text{if $a-b-1=0$}.
\endaligned
\right.
$$
If $a\pm b\pm 1=0$ and $ab=0$, then $X$ is singular along the~curve
$$
\big\{w=0, x=z, y=t, u_1z=v_1t, u_2z=v_2t, v_1u_2=u_1v_2\big\}.
$$

Let $Q$ be the~quadric threefold in $\mathbb{P}^4$ given by $w^2+xy+zt=a(xt+yz)+b(xz+yt)$,
and let $\pi\colon{X}\to{Q}$ be the~morphism that is induced by the~natural projection ${X}\to\mathbb{P}^4$.
Then $\pi$ is a~blow up of the~following two smooth conics:
\begin{align*}
{C}_1&=\{w^2+zt=0,x=0,y=0\},\\
{C}_2&=\{w^2+xy=0,z=0,t=0\}.
\end{align*}
The quadric $Q$ is smooth $\iff$  $a\pm b\pm 1=0$.
If $a\pm b\pm 1=0$ and $ab\ne 0$, then
$$
\mathrm{Sing}(Q)=\left\{\aligned
&[1:1:-1:-1:0]\ \text{if $a+b+1=0$}, \\
&[1:1:1:1:0]\ \text{if $a+b-1=0$}, \\
&[1:-1:1:-1:0]\ \text{if $a-b+1=0$}, \\
&[1:-1:-1:1:0]\ \text{if $a-b-1=0$}.
\endaligned
\right.
$$
If $a\pm b\pm 1=0$ and $ab\ne 0$, then $Q$ is singular along the~line $\{w=0, x=z, y=t\}$.

The threefold $X$ admits an action of the~group $\mumu_2^3$ given by
\begin{align*}
\tau_1\colon \big([u_1:v_1],[u_2:v_2],[x:y:z:t:w]\big) &\mapsto \big([v_1:u_1],[v_2:u_2],[y:x:t:z:w]\big),\\
\tau_2\colon \big([u_1:v_1],[u_2:v_2],[x:y:z:t:w]\big) &\mapsto \big([u_2:v_2],[u_1:v_1],[z:t:x:y:w]\big),\\
\tau_3\colon \big([u_1:v_1],[u_2:v_2],[x:y:z:t:w]\big) &\mapsto \big([u_1:v_1],[u_2:v_2],[x:y:z:t:-w]\big).
\end{align*}
Let $G$ be the~subgroup in $\mathrm{Aut}(X)$ generated by the~involutions $\tau_1$, $\tau_2$, $\tau_3$.
Then
$$
G\cong\mumu_2^3,
$$
the~blow up $\pi\colon X\to Q$ is $G$-equivariant,
and the~$G$-action on $Q$ induces a~$G$-action on the~projective space $\mathbb{P}^4$.
Therefore, we can also consider $G$ as a~subgroup in $\mathrm{PGL}_4(\mathbb{C})$,
where $\tau_1$, $\tau_2$, $\tau_3$ act on  $\mathbb{P}^4$ as in Section~\ref{section:GIT-2}:
\begin{align*}
\tau_1\colon [x:y:z:t:w] &\mapsto [y:x:t:z:w],\\
\tau_2\colon [x:y:z:t:w] &\mapsto [z:t:x:y:w],\\
\tau_3\colon [x:y:z:t:w] &\mapsto [x:y:z:t:-w].
\end{align*}
Observe that the~only $G$-fixed points in  $\mathbb{P}^4$ are
$$
[1:1:-1:-1:0],[1:1:1:1:0],[1:-1:1:-1:0],[1:-1:-1:1:0].
$$

\begin{lemma}
\label{lemma:G-action}
The following two assertions hold:
\begin{itemize}
\item[$(1)$] if $X$ is smooth, then $X$ does not have $G$-fixed points;
\item[$(2)$] if $X$ has one singular point, then $\mathrm{Sing}(X)$ is the~only $G$-fixed point in $X$.
\end{itemize}
\end{lemma}

\begin{proof}
If $Q$ is smooth, it does not contain any $G$-fixed points.
Vice versa, if $Q$ is singular, then its singular point is the~only $G$-fixed point in $Q$.
The claim follows.
\end{proof}

We will also need the~following technical lemma:

\begin{lemma}
\label{lemma:ODP-G-action}
Suppose that $a\pm b\pm 1=0$ and $ab\ne 0$. Let $O$ be the~singular point $\mathrm{Sing}(X)$,
let $\varphi\colon\widetilde{X}\to X$ be the~blow up of the~point $O$, and let $F$ be the~$\varphi$-exceptional~surface.
Then
\begin{itemize}
\item[$(1)$] $\varphi$ is $G$-equivariant,
\item[$(2)$] the~group $G$ acts faithfully on $F$,
\item[$(3)$] $\mathrm{rk}\,\mathrm{Pic}^G(F)=1$,
\item[$(4)$] $F$ has no $G$-fixed points.
\end{itemize}
\end{lemma}

\begin{proof}
We only consider the~case when $b=1-a$ and $a\not\in\{0,1\}$; the other cases are similar. Then
$$
Q=\big\{w^2+xy+zt-a(xt+yz)-(1-a)(xz+yt)=0\big\},
$$
and $\pi(O)=[1:1:1:1:0]$.
Let us introduce new coordinates on $\mathbb{P}^4$ as follows:
$$
\left\{\aligned
&\mathbf{x}=x-y, \\
&\mathbf{y}=x+y-z-t, \\
&\mathbf{z}=z-t, \\
&\mathbf{t}=x+y+z+t, \\
&\mathbf{w}=w.
\endaligned
\right.
$$
In these coordinates, $Q=\{\mathbf{x}^2-\mathbf{y}^2+\mathbf{z}^2=(4a-2)\mathbf{x}\mathbf{z}+4\mathbf{w}^2\}$ and $\pi(O)=[0:0:0:1:0]$.
Moreover, the~involutions $\tau_1$, $\tau_2$, $\tau_3$ acts as follows:
\begin{align*}
\tau_1\colon \big[\mathbf{x}:\mathbf{y}:\mathbf{z}:\mathbf{t}:\mathbf{w}\big]&\mapsto\big[-\mathbf{x}:\mathbf{y}:-\mathbf{z}:\mathbf{t}:\mathbf{w}\big],\\
\tau_2\colon \big[\mathbf{x}:\mathbf{y}:\mathbf{z}:\mathbf{t}:\mathbf{w}\big]&\mapsto\big[\mathbf{z}:-\mathbf{y}:\mathbf{x}:\mathbf{t}:\mathbf{w}\big],\\
\tau_3\colon \big[\mathbf{x}:\mathbf{y}:\mathbf{z}:\mathbf{t}:\mathbf{w}\big]&\mapsto\big[\mathbf{x}:\mathbf{y}:\mathbf{z}:\mathbf{t}:-\mathbf{w}\big].
\end{align*}
Furthermore, we can $G$-equivariantly identify $F$ with the~projectivization of the~tangent cone to $Q$ at the~point $\pi(O)$.
So, we can consider $F$ as the~quadric surface in $\mathbb{P}^3$ given by
$$
\mathbf{x}^2-\mathbf{y}^2+\mathbf{z}^2=(4a-2)\mathbf{x}\mathbf{z}+4\mathbf{w}^2,
$$
where now we consider $[\mathbf{x}:\mathbf{y}:\mathbf{z}:\mathbf{w}]$ as coordinates on $\mathbb{P}^3$.
Now, all assertions are easy to check looking on how the~group $G$ acts on the~surface $F$.
\end{proof}

We have the~following $G$-equivariant commutative diagram
$$
\xymatrix{
&X\ar@{->}[dl]_{\pi}\ar@{->}[dr]^{\eta}&\\%
{Q}\ar@{-->}[rr]_{\rho}&&\mathbb{P}^1\times\mathbb{P}^1}
$$
where $\eta$ is a~conic bundle given by the~natural projection $X\to\mathbb{P}^1\times\mathbb{P}^1$,
and $\rho$ is a~rational map that is given by $[x:y:z:t:w]\mapsto([x:y],[z:t])$.

\begin{remark}
\label{remark:G-action-P1-P1}
The~only $G$-fixed points in $\mathbb{P}^1\times\mathbb{P}^1$ are $([1:1],[1:1])$ and $([-1:1],[-1:1])$.
\end{remark}

Let $\Delta$ be the~discriminant curve of the~conic bundle $\eta$.
Then $\Delta$ is given by
\begin{multline*}
\quad\quad\quad a^2u_1^2v_2^2+2abu_1^2u_2v_2+b^2u_1^2u_2^2+(2a^2+2b^2-4)v_1v_2u_1u_2+\\
+2abu_1v_1u_2^2+2abu_1v_1v_2^2+a^2u_2^2v_1^2+2abu_2v_1^2v_2+b^2v_1^2v_2^2=0,\quad\quad\quad
\end{multline*}
where we consider $([u_1:v_1],[u_2:v_2])$ as coordinates on $\mathbb{P}^1\times\mathbb{P}^1$.

\begin{remark}
\label{remark:conic-bundle-1}
If $a\pm b\pm 1\ne 0$, then $\Delta$ is smooth.
If $a\pm b\pm 1=0$ and $ab\ne 0$, then
$$
\mathrm{Sing}(\Delta)=\left\{\aligned
&\big([1:1],[1:1]\big)\ \text{if $a+b\pm 1=0$}, \\
&\big([-1:1],[-1:1]\big)\ \text{if $a-b\pm 1=0$}. \\
\endaligned
\right.
$$
\end{remark}

Let $H$ be a~proper transform on $X$ of a~general hyperplane section of the~quadric $Q$, and let $E_1$ and $E_2$ be the~$\pi$-exceptional divisors
such that $\pi(E_1)=C_1$ and $\pi(E_2)=C_2$.

\begin{lemma}
\label{lemma:class-groups}
Suppose that $X$ has at most isolated singularities.
Then
$$
\mathrm{Pic}^G(X)=\mathrm{Cl}^G(X)=\mathbb{Z}[H]\oplus\mathbb{Z}[E_1+E_2].
$$
\end{lemma}

\begin{proof}
If the~threefold $X$ is smooth, the~assertion is obvious, since $G$ swaps $E_1$ and $E_2$.
If $X$ has one singular point, the~assertion follows from Lemma~\ref{lemma:ODP-G-action}.
\end{proof}

Now, let us describe the~cone of $G$-invariant effective divisors on $X$.

\begin{lemma}
\label{lemma:Eff-cone}
Suppose that $X$ has at most isolated singularities.
Let $S$ be an $G$-invariant surface in $X$.
Then
$S\sim n_1(E_1+E_2)+n_2(2H-E_1-E_2)+n_3H$ for some \mbox{$n_1,n_2,n_3\in\mathbb{Z}_{\geqslant 0}$}.
\end{lemma}

\begin{proof}
By Lemma~\ref{lemma:class-groups}, we see that $S\sim k_1H+k_2(E_1+E_2)$ for some integers $k_1$ and $k_2$.
On the~other hand, the~conic bundle $\eta\colon X\to\mathbb{P}^1\times\mathbb{P}^1$ is given by $|2H-E_1-E_2|$.
Then
$$
S\sim_{\mathbb{Q}} m_1(E_1+E_2)+m_2(2H-E_1-E_2)
$$
for some non-negative rational numbers $m_1$ and $m_2$.
We have $k_1=2m_2$ and $m_1-m_2=k_2$.
If $k_1$ is even, then $m_1$ and $m_2$ are integers and we are done, since $\mathrm{Pic}(X)$ has no torsion.
So, we may assume that $k_1=2n+1$, where $n\in\mathbb{Z}$.
Then $m_2=n+\frac{1}{2}$ and $m_1=k_2+n+\frac{1}{2}$,
which gives $n\geqslant 0$ and $k_2+n\geqslant 0$, since $m_1\geqslant 0$ and $m_2\geqslant 0$.
Hence, we have
$$
S\sim_{\mathbb{Q}} (k_2+n)(E_1+E_2)+n(2H-E_1-E_2)+H,
$$
which gives $S\sim (k_2+n)(E_1+E_2)+n(2H-E_1-E_2)+H$, since $\mathrm{Pic}(X)$ has no torsion.
\end{proof}

\begin{corollary}
\label{corollary:divisorial-stability}
Suppose that $X$ has at most isolated singularities.
Let $S$ be a~$G$-invariant irreducible surface in $X$. Then $\beta(S)>0$.
\end{corollary}

\begin{proof}
We have $\beta(S)=1-S_X(S)$, and it follows from Lemma~\ref{lemma:Eff-cone} that
$$
S_X(S)=\frac{1}{26}\int\limits_{0}^{\infty}\mathrm{vol}\big(-K_X-uS\big)du\leqslant\frac{1}{26}\int\limits_{0}^{\infty}\mathrm{vol}\big(-K_X-uF\big)du,
$$
where $F$ is one divisor among $H$, $E_1+E_2$, $2H-E_1-E_2$.
On the~other hand, we have
$$
\frac{1}{26}\int\limits_{0}^{\infty}\mathrm{vol}\big(-K_X-uH\big)du=\frac{1}{26}\int\limits_{0}^{1}\big(-K_X-uH\big)^3du=\frac{1}{26}\int\limits_{0}^{1}2(1-u)(u^2-8u+13)du=\frac{21}{52}.
$$
Similarly, we have
$$
\frac{1}{26}\int\limits_{0}^{\infty}\mathrm{vol}\big(-K_X-u(E_1+E_2)\big)du=\frac{1}{26}\int\limits_{0}^{\frac{1}{2}}2(2u-1)(2u^2-2u-13)du=\frac{53}{208}.
$$
Finally, we compute
\begin{multline*}
\frac{1}{26}\int\limits_{0}^{\infty}\mathrm{vol}\big(-K_X-u(2H-E_1-E_2)\big)du=
\frac{1}{26}\int\limits_{0}^{1}\big((3-2u)H+(u-1)E_1+(u-1)E_2\big)^3du+\\
+\frac{1}{26}\int\limits_{1}^{\frac{3}{2}}\big((3-2u)H\big)^3du=\frac{1}{26}\int\limits_{0}^{1}12u^2-36u+26du+\frac{1}{26}\int\limits_{1}^{\frac{3}{2}}2(3-2u)^3du=\frac{49}{104}.
\end{multline*}
Thus, we conclude that $\beta(S)>0$, as claimed.
\end{proof}

We conclude this section with the~following technical result:

\begin{proposition}
\label{proposition:delta-dP4}
Suppose that $X$ has at most isolated singularities, and let $S$ be a~smooth surface in one of the linear systems $|H-E_1|$ or $|H-E_2|$. Then
$$
\delta_P(X)\geqslant\frac{104}{99}
$$
for every point $P\in S$ such that $P\not\in E_1\cup E_2$.
\end{proposition}

\begin{proof}
We may assume that $S\in |H-E_1|$.
Let $u$ be a~non-negative real number. Then
$$
-K_X-uS\sim_{\mathbb{R}}(2-u)S+(H-E_2)+E_1\sim_{\mathbb{R}} (3-u)H-(1-u)E_1-E_2.
$$
Then $-K_X-uS$ is nef $\iff$ $u\in[0,1]$, and $-K_X-uS$ is pseudo-effective $\iff$ $u\in[0,2]$.
Moreover, if $u\in[1,2]$, then its Zariski decomposition can be described as follows:
$$
-K_X-uS\sim_{\mathbb{R}}\underbrace{(3-u)H-E_2}_{\text{positive part}}+\underbrace{(u-1)E_1}_{\text{negative part}}.
$$
Thus, for transparency, we let
$$
P(u)=\left\{\aligned
&(3-u)H-(1-u)E_1-E_2\ \text{if $0\leqslant u\leqslant 1$}, \\
&(3-u)H-E_2\ \text{if $1\leqslant u\leqslant 2$},
\endaligned
\right.
$$
and
$$
N(u)=\left\{\aligned
&0\ \text{if $0\leqslant u\leqslant 1$}, \\
&(u-1)E_1\ \text{if $1\leqslant u\leqslant 2$}.
\endaligned
\right.
$$
Then
$$
\big(P(u)\big)^3=\left\{\aligned
&26-18u\ \text{if $0\leqslant u\leqslant 1$}, \\
&40-2u^3+18u^2-48u\ \text{if $1\leqslant u\leqslant 2$}.
\endaligned
\right.
$$
Then
$$
S_X(S)=\frac{1}{(-K_{X})^3}\int\limits_{0}^{2}\mathrm{vol}\big(-K_X-uS\big)du=\frac{1}{26}\int\limits_{0}^{2}\big(P(u)\big)^3du=\frac{3}{4},
$$
so that $\beta(S)=1-S_X(S)=\frac{1}{4}$.

Let $P$ be a~point in $S$, and let $C$ be an irreducible smooth curve in $S$ such that $P\in C$.
Write $N(u)\vert_{S}=N^\prime(u)+\mathrm{ord}_{C}(N(u)\vert_{S})C$.
For every $u\in[0,2]$, set
$$
t(u)=\mathrm{sup}\Big\{v\in\mathbb{R}_{\geqslant 0}\ \big\vert\ \text{the~divisor  $P(u)\big\vert_{S}-vC$ is pseudo-effective}\Big\}.
$$
Let $v$ be a~real number in $[0,t(u)]$, let $P(u,v)$ be the~positive part of the~Zariski decomposition of~the~$\mathbb{R}$-divisor $P(u)\vert_{S}-vC$,
and let $N(u,v)$ be its~negative part. Set
$$
S\big(W^S_{\bullet,\bullet};C\big)=\frac{3}{(-K_X)^3}\int\limits_0^2\big(P(u)\big\vert_{S}\big)^2\cdot\mathrm{ord}_{C}\big(N(u)\big\vert_{S}\big)du+\frac{3}{(-K_X)^3}\int\limits_0^2\int\limits_0^{t(u)}\big(P(u,v)\big)^2dvdu
$$
and
\begin{multline*}
S\big(W_{\bullet, \bullet,\bullet}^{S,C};P\big)=\frac{3}{(-K_X)^3}\int\limits_0^2\int\limits_0^{t(u)}\big(P(u,v)\cdot C\big)^2dvdu+\\
+\frac{6}{(-K_X)^3} \int\limits_0^2\int\limits_0^{t(u)}\big(P(u,v)\cdot C\big)\cdot \mathrm{ord}_P\big(N^\prime(u)\big|_C+N(u,v)\big|_C\big)dvdu.
\end{multline*}
Then it follows from \cite{AbbanZhuang,Book} that
$$
\delta_P(X)\geqslant\min\Bigg\{\frac{1}{S_X(S)},\frac{1}{S\big(W^S_{\bullet,\bullet};C\big)},\frac{1}{S\big(W_{\bullet, \bullet,\bullet}^{S,C};P\big)}\Bigg\}.
$$
But $S_X(S)=\frac{3}{4}$.
Hence, to complete the~proof,
it is enough to find an~irreducible smooth curve $C\subset S$ such that $P\in C$,
and $S(W^S_{\bullet,\bullet};C)\leqslant\frac{99}{104}\geqslant S(W_{\bullet, \bullet,\bullet}^{S,C};P)$ if $P\not\in E_1\cup E_2$.

Now, suppose that $P\not\in E_1\cup E_2$.
In particular, this assumption implies $P\not\in\mathrm{Supp}(N(u))$, so that the~formulas for $S(W^S_{\bullet,\bullet};C)$ and $S(W_{\bullet, \bullet,\bullet}^{S,C};P)$ simplify as follows:
$$
S\big(W^S_{\bullet,\bullet};C\big)=\frac{3}{(-K_X)^3}\int\limits_0^2\int\limits_0^{t(u)}\big(P(u,v)\big)^2dvdu
$$
and
$$
S\big(W_{\bullet, \bullet,\bullet}^{S,C};P\big)=\frac{3}{(-K_X)^3}\int\limits_0^2\int\limits_0^{t(u)}\big(P(u,v)\cdot C\big)\Big
(\big(P(u,v)\cdot C\big)+2\mathrm{ord}_P\big(N(u,v)\big|_C\big)\Big)dvdu.
$$
To find the~required curve $C$, recall that $S$ is a~del Pezzo surface of degree $6$, and
$$
E_2\big\vert_{S}=\mathbf{e}_1+\mathbf{e}_2,
$$
where $\mathbf{e}_1$ and $\mathbf{e}_2$ are disjoint $(-1)$-curves in the~surface $S$.
Let $Z$ be the~fiber of the~conic bundle $S\to\mathbb{P}^1$ given by $([u_1:v_1],[u_2:v_2],[x:y:z:t:w])\mapsto[u_2:v_2]$ such that $P\in Z$.
Then $Z\cdot \mathbf{e}_1=Z\cdot\mathbf{e}_2=1$ and
$$
E_1\big\vert_{S}\sim H\big\vert_{S}\sim Z+\mathbf{e}_1+\mathbf{e}_2.
$$
Let us choose $C$ to be an irreducible component of the~fiber $Z$ that contains the~point $P$.
A priori, we may have the~following two possibilities:
\begin{enumerate}
\item $Z$ is an irreducible smooth rational curve and $Z^2=0$;
\item $Z=\ell_1+\ell_2$ for two $(-1)$-curves $\ell_1$ and $\ell_2$ such that
$\ell_1\cdot\ell_2=\ell_1\cdot \mathbf{e}_1=\ell_2\cdot\mathbf{e}_2=1$.
\end{enumerate}
In the~first case, we have $C=Z$.
In the~second case, we may assume $P\in\ell_1$, so~$C=\ell_1$.
Let us compute $S(W^S_{\bullet,\bullet};C)$ and $S(W_{\bullet, \bullet,\bullet}^{S,C};P)$ in each of these cases.

First, we note that
$$
P(u)\big\vert_{S}=\left\{\aligned
&2Z+\mathbf{e}_1+\mathbf{e}_2\ \text{if $0\leqslant u\leqslant 1$}, \\
&(3-u)Z+(2-u)(\mathbf{e}_1+\mathbf{e}_2)\ \text{if $1\leqslant u\leqslant 2$}.
\endaligned
\right.
$$

Let $v$ be a~non-negative real number. If $C=Z$, then
$$
P(u)\big\vert_{S}-vC=\left\{\aligned
&(2-v)C+\mathbf{e}_1+\mathbf{e}_2\ \text{if $0\leqslant u\leqslant 1$}, \\
&(3-u-v)C+(2-u)(\mathbf{e}_1+\mathbf{e}_2)\ \text{if $1\leqslant u\leqslant 2$},
\endaligned
\right.
$$
which implies that
$$
t(u)=\left\{\aligned
&2\ \text{if $0\leqslant u\leqslant 1$}, \\
&3-u\ \text{if $1\leqslant u\leqslant 2$}.
\endaligned
\right.
$$
Furthermore, if $C=Z$ and $u\in[0,1]$, then
$$
P(u,v)=\left\{\aligned
&(2-v)C+\mathbf{e}_1+\mathbf{e}_2\ \text{if $0\leqslant v\leqslant 1$}, \\
&(2-v)\big(C+\mathbf{e}_1+\mathbf{e}_2\big)\ \text{if $1\leqslant v\leqslant 2$},
\endaligned
\right.
$$
and
$$
N(u,v)=\left\{\aligned
&0\ \text{if $0\leqslant v\leqslant 1$}, \\
&(v-1)\big(\mathbf{e}_1+\mathbf{e}_2\big)\ \text{if $1\leqslant v\leqslant 2$},
\endaligned
\right.
$$
which gives
$$
\big(P(u,v)\big)^2=\left\{\aligned
&6-4v\ \text{if $0\leqslant v\leqslant 1$}, \\
&2(v-2)^2\ \text{if $1\leqslant v\leqslant 2$},
\endaligned
\right.
$$
and
$$
P(u,v)\cdot C=\left\{\aligned
&2\ \text{if $0\leqslant v\leqslant 1$}, \\
&4-2v\ \text{if $1\leqslant v\leqslant 2$}.
\endaligned
\right.
$$
Similarly, if $C=Z$ and $u\in[1,2]$, then
$$
P(u,v)=\left\{\aligned
&(3-u-v)C+(2-u)(\mathbf{e}_1+\mathbf{e}_2)\ \text{if $0\leqslant v\leqslant 1$}, \\
&(3-u-v)\big(C+\mathbf{e}_1+\mathbf{e}_2\big)\ \text{if $1\leqslant v\leqslant 3-u$},
\endaligned
\right.
$$
and
$$
N(u,v)=\left\{\aligned
&0\ \text{if $0\leqslant v\leqslant 1$}, \\
&(v-1)\big(\mathbf{e}_1+\mathbf{e}_2\big)\ \text{if $1\leqslant v\leqslant 3-u$},
\endaligned
\right.
$$
so that
$$
\big(P(u,v)\big)^2=\left\{\aligned
&2(2-u)(4-u-2v)\ \text{if $0\leqslant v\leqslant 1$}, \\
&2(3-u-v)^2\ \text{if $1\leqslant v\leqslant 3-u$},
\endaligned
\right.
$$
and
$$
P(u,v)\cdot C=\left\{\aligned
&4-2u\ \text{if $0\leqslant v\leqslant 1$}, \\
&6-2u-2v\ \text{if $1\leqslant v\leqslant 3-u$}.
\endaligned
\right.
$$
In particular, if $C=Z$, then $P\not\in\mathrm{Supp}(N(u,v))$, because $P\not\in E_2$ by our assumption.
Now, integrating, we get $S(W^S_{\bullet,\bullet};C)=\frac{3}{4}<\frac{99}{104}$ and $S(W_{\bullet, \bullet,\bullet}^{S,C};P)=\frac{21}{26}<\frac{99}{104}$.

Hence, to complete the~proof, we may assume that $Z=\ell_1+\ell_2$ and $C=\ell_1$. Then
$$
P(u)\big\vert_{S}-vC=\left\{\aligned
&(2-v)C+2\ell_2+\mathbf{e}_1+\mathbf{e}_2\ \text{if $0\leqslant u\leqslant 1$}, \\
&(3-u-v)C+(3-u)\ell_2+(2-u)(\mathbf{e}_1+\mathbf{e}_2)\ \text{if $1\leqslant u\leqslant 2$},
\endaligned
\right.
$$
which implies that
$$
t(u)=\left\{\aligned
&2\ \text{if $0\leqslant u\leqslant 1$}, \\
&3-u\ \text{if $1\leqslant u\leqslant 2$}.
\endaligned
\right.
$$
Further, if $u\in[0,1]$, then
$$
P(u,v)=\left\{\aligned
&(2-v)C+2\ell_2+\mathbf{e}_1+\mathbf{e}_2\ \text{if $0\leqslant v\leqslant 1$}, \\
&(2-v)\big(C+\mathbf{e}_1\big)+(3-v)\ell_2+\mathbf{e}_2\ \text{if $1\leqslant v\leqslant 2$},
\endaligned
\right.
$$
and
$$
N(u,v)=\left\{\aligned
&0\ \text{if $0\leqslant v\leqslant 1$}, \\
&(v-1)\big(\mathbf{e}_1+\ell_2\big)\ \text{if $1\leqslant v\leqslant 2$},
\endaligned
\right.
$$
which gives
$$
\big(P(u,v)\big)^2=\left\{\aligned
&6-v^2-2v\ \text{if $0\leqslant v\leqslant 1$}, \\
&(2-v)(4-v)\ \text{if $1\leqslant v\leqslant 2$},
\endaligned
\right.
$$
and
$$
P(u,v)\cdot C=\left\{\aligned
&1+v\ \text{if $0\leqslant v\leqslant 1$}, \\
&3-v\ \text{if $1\leqslant v\leqslant 2$}.
\endaligned
\right.
$$
Likewise, if $u\in[1,2]$, then
$$
P(u,v)=\left\{\aligned
&(3-u-v)C+(3-u)\ell_2+(2-u)(\mathbf{e}_1+\mathbf{e}_2)\ \text{if $0\leqslant v\leqslant 2-u$}, \\
&(3-u-v)C+(5-2u-v)\ell_2+(2-u)(\mathbf{e}_1+\mathbf{e}_2)\ \text{if $2-u\leqslant v\leqslant 1$}, \\
&(3-u-v)(C+\mathbf{e}_1)+(5-2u-v)\ell_2+(2-u)\mathbf{e}_2\ \text{if $1\leqslant v\leqslant 3-u$},
\endaligned
\right.
$$
and
$$
N(u,v)=\left\{\aligned
&0\ \text{if $0\leqslant v\leqslant 2-u$}, \\
&(v-2+u)\ell_2\ \text{if $2-u\leqslant v\leqslant 1$}, \\
&(v-2+u)\ell_2+(v-1)\mathbf{e}_1\ \text{if $1\leqslant v\leqslant 3-u$},
\endaligned
\right.
$$
so that
$$
\big(P(u,v)\big)^2=\left\{\aligned
&2u^2+2uv-v^2-12u-4v+16\ \text{if $0\leqslant v\leqslant 2-u$}, \\
&(u-2)(3u+4v-10)\ \text{if $2-u\leqslant v\leqslant 1$}, \\
&(3-u-v)(7-3u-v)\ \text{if $1\leqslant v\leqslant 3-u$},
\endaligned
\right.
$$
and
$$
P(u,v)\cdot C=\left\{\aligned
&2-u+v\ \text{if $0\leqslant v\leqslant 2-u$}, \\
&4-2u\ \text{if $2-u\leqslant v\leqslant 1$}, \\
&5-2u-v\ \text{if $1\leqslant v\leqslant 3-u$}.
\endaligned
\right.
$$
Now, integrating, we get $S(W^S_{\bullet,\bullet};C)=\frac{99}{104}$.
Note that $P\not\in\mathbf{e}_1$, since $P\not\in E_2$ by assumption.
Therefore, if $P\not\in\ell_2$, then $P\not\in\mathrm{Supp}(N(u,v))$,
which implies that $S(W_{\bullet, \bullet,\bullet}^{S,C};P)=\frac{37}{52}<\frac{99}{104}$.
Similarly, if $P\in\ell_2$, then
\begin{multline*}
S\big(W_{\bullet, \bullet,\bullet}^{S,C};P\big)=\frac{37}{52}+\frac{6}{26} \int\limits_0^2\int\limits_0^{t(u)}\big(P(u,v)\cdot C\big)\cdot \mathrm{ord}_P\big(N(u,v)\big|_C\big)dvdu=\\
=\frac{37}{52}+\frac{6}{26}\int\limits_0^1\int\limits_1^{2}\big(P(u,v)\cdot C\big)(v-1)dvdu+\frac{6}{26}\int\limits_1^2\int\limits_{2-u}^{3-u}\big(P(u,v)\cdot C\big)(v-2+u)dvdu=\\
=\frac{37}{52}+\frac{6}{26}\int\limits_0^1\int\limits_1^{2}(3-v)(v-1)dvdu+\frac{6}{26}\int\limits_1^2\int\limits_{2-u}^{1}(4-2u)(v-2+u)dvdu+\\
+\frac{6}{26}\int\limits_1^2\int\limits_1^{3-u}(5-2u-v)(v-2+u)dvdu=\frac{99}{104}.\quad\quad\quad\quad\quad\quad\quad\quad\quad
\end{multline*}
This completes the~proof of Proposition~\ref{proposition:delta-dP4}.
\end{proof}

\subsection{Smooth threefolds.}
\label{section:beth-smooth}
We continue to use the notation introduced earlier in this section.
The goal of this subsection is to give a~new proof~of the~following result.

\begin{theorem}[{\cite[Proposition 5.79]{Book}}]
\label{theorem:K-stability-smooth}
Let $X$ be a smooth Fano threefold from class $(\beth)$. Then $X$ is K-polystable.
\end{theorem}

We begin with a technical lemma that will also be useful in the singular case.

\begin{lemma}
\label{lemma:G-invariant-center}
Suppose that $X$ has at worst isolated singularities.
Let $\mathbf{F}$ be a~$G$-invariant prime divisor $\mathbf{F}$ over the~threefold $X$ with
$$
\beta(\mathbf{F})=A_X(\mathbf{F})-S_X(\mathbf{F})\leqslant 0,
$$
and let $Z$ be its center on $X$. Suppose that $Z$ is not a singular point of the~threefold $X$. Then $Z$ is an irreducible curve with $\eta(Z)=([1:1],[1:1])$ or $\eta(Z)=([1:-1],[1:-1])$.
\end{lemma}

\begin{proof}
Note first that $Z$ is a~$G$-invariant curve, by Lemma~\ref{lemma:G-action} and Corollary~\ref{corollary:divisorial-stability}, and
$$
\delta_P(X)\leqslant 1
$$
for every point $P\in Z$. Observe also that $Z\not\subset E_1\cup E_2$, because the~surfaces $E_1$ and $E_2$ are disjoint and swapped by the~action of the~group $G$.

Let $S$ be any surface in $|H-E_1|$.
If $S\cdot Z\ne 0$, then choosing $S$ to be general enough,
we see that $S$ is a~smooth del Pezzo surface and $S\cap Z$ contains a~point $P\not\in E_1\cup E_2$.
This~would give $\delta_P(X)>1$ by Proposition~\ref{proposition:delta-dP4}, which is a~contradiction.
So we have
$$
\big(H-E_1\big)\cdot Z=0.
$$
Similarly, we get $(H-E_2)\cdot Z=0$.
Therefore, we conclude that $\eta(Z)$ is a~point in $\mathbb{P}^1\times\mathbb{P}^1$. Finally, using Remark~\ref{remark:G-action-P1-P1} we get $\eta(Z)=([1:1],[1:1])$ or $\eta(Z)=([1:-1],[1:-1])$.
\end{proof}

\begin{proof}[Proof of Theorem \ref{theorem:K-stability-smooth}]
Suppose for a contradiction that the Fano threefold $X$ is smooth, but $X$ is not K-polystable.
It follows from \cite{Fujita2019,Li,Zhuang}~that there exists a~$G$-invariant prime divisor $\mathbf{F}$ over $X$
such that $ \beta(\mathbf{F})\leqslant 0$.
Let $Z$ be the center of $\mathbf{F}$ on $X$. Then by Lemma~\ref{lemma:G-invariant-center}, the center $Z$ is a curve with $\eta(Z)=([1:1],[1:1])$ or $\eta(Z)=([1:-1],[1:-1])$.

Let $S$ be the~unique surface in the~pencil $|H-E_1|$ that contains $Z$, and let $\ell=\eta(S)$.
If $\eta(Z)=([1:\pm 1],[1:\pm 1])$, then $\ell=\{u_1\pm v_1=0\}$.
Moreover, one can check that
\begin{itemize}
\item $\{u_1-v_1=0\}$ intersects $\Delta$ in a~single point $\iff$  $b+a\pm 1=0$,
\item $\{u_1+v_1=0\}$ intersects $\Delta$ in a~single point $\iff$  $b-a\pm 1=0$.
\end{itemize}
Since $Q$ is smooth by assumption, we have $b\pm a\pm 1\ne 0$,
and $\ell$ intersects the~discriminant curve $\Delta$ in two distinct points.
This implies that $S$ is a~smooth del Pezzo surface.

Now, applying Proposition~\ref{proposition:delta-dP4}, we get $\delta_P(X)>1$ for a~general point $P$ in the~curve $Z$,
which contradicts $\beta(\mathbf{F})\leqslant 0$. This completes the~proof of Theorem~\ref{theorem:K-stability-smooth}.
\end{proof}

\subsection{Threefolds with one singular point.}
\label{section:beth-singular}

We continue to use the notation introduced at the beginning of this section.
In this subsection, we will  prove the~following result.

\begin{theorem}
\label{theorem:K-stability-isolated}
Suppose that $X$ is a threefold from class $(\beth)$ with one singular point. Then $X$ is K-polystable.
\end{theorem}

Suppose that the~threefold $X$ has one singular point. Then $a\pm b\pm 1=0$ and $ab\ne 0$.
Up to a~change of coordinates, we may assume that $b=1-a$ and $a\not\in\{0,1\}$.

Let $O=\mathrm{Sing}(X)$. Recall  that
$$
X=\big\{u_1x=v_1y,u_2z=v_2t,w^2+xy+zt=a(xt+yz)+(1-a)(xz+yt)\big\}\subset\mathbb{P}^1\times\mathbb{P}^1\times\mathbb{P}^4,
$$
and $O=([1:1],[1:1],[1:1:1:1:0])$. Note that $O$ is an isolated threefold node ($\mathrm{cA}_1$) singularity. Similarly, we have
$$
Q=\big\{w^2+xy+zt=a(xt+yz)+(1-a)(xz+yt)\big\}\subset\mathbb{P}^4,
$$
and $\mathrm{Sing}(Q)=\pi(O)=[1:1:1:1:0]$. Note that $\eta(O)=([1:1],[1:1])$.

The fiber of the~conic bundle $\eta$ over the~point $([1:1],[1:1])$ is a~reducible reduced conic $L_1+L_2$,
where $L_1$ and $L_2$ are two smooth rational curves such that $L_1\cap L_2=O$,
and $\pi(L_1)$ and $\pi(L_2)$ are the~only lines in $Q$ that intersect  $C_1$ and $C_2$.
We have
\begin{align*}
\pi(L_1)&=\{x=y,z=t,w+x-z=0\},\\
\pi(L_2)&=\{x=y,z=t,w-x+z=0\}.
\end{align*}
We will use the~curves $L_1$ and $L_2$ in the~proofs of the~next two lemmas.

\begin{lemma}
\label{lemma:K-stability-isolated-1}
Let $\mathbf{F}$ be a~$G$-invariant prime divisor $\mathbf{F}$ over the~threefold $X$ with $\beta(\mathbf{F})\leqslant 0$,
and let $Z$ be its center on the~threefold $X$. Then~$Z=O$.
\end{lemma}

\begin{proof}
Suppose that $Z\ne O$. Applying Lemma \ref{lemma:G-invariant-center},
we find that $Z$ is an irreducible curve such that $\eta(Z)=([1:1],[1:1])$ or $\eta(Z)=([1:-1],[1:-1])$.
But $Z\ne L_1$ and $Z\ne L_2$, since neither of the~curves $L_1$ nor $L_2$ is $G$-invariant, so  $\eta(Z)=([1:-1],[1:-1])$.

Let $S$ be the~surface in $|H-E_1|$ that is cut out on $X$ by the~equation $u_1+v_1=0$.
Then $S$ is smooth and $Z\subset S$,
so $\delta_P(X)>1$ for a~general point $P\in Z$ by Proposition~\ref{proposition:delta-dP4}.
This contradicts $\beta(\mathbf{F})\leqslant 0$.
\end{proof}

\begin{lemma}
\label{lemma:K-stability-isolated-2}
Let $\mathbf{F}$ be any $G$-invariant prime divisor $\mathbf{F}$ over $X$ such that $C_X(\mathbf{F})=O$.
Then $\beta(\mathbf{F})>0$.
\end{lemma}

\begin{proof}
Let $\varphi\colon\widetilde{X}\to X$ and $\phi\colon\widetilde{Q}\to Q$ be blow ups of the~points $O$ and $\pi(O)$, respectively.
As~in the~proof of Lemma~\ref{lemma:K-unstable}, we have the~following $G$-equivariant commutative diagram:
$$
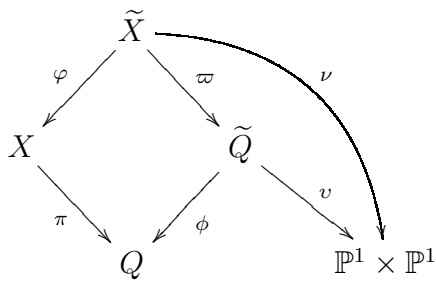
\xymatrix{
&\widetilde{X}\ar@{->}[dl]_{\varphi}\ar@{->}[dr]^{\varpi}\ar@/^2.5pc/@{->}[ddrr]^{\nu}&\\%
X\ar@{->}[dr]_{\pi}&&\widetilde{Q}\ar@{->}[dl]^{\phi}\ar@{->}[dr]^{\upsilon}&\\%
&{Q}&&\mathbb{P}^1\times\mathbb{P}^1}
$$
where $\varpi$ is the~blow up of the~proper transforms of the~conics $C_1$ and $C_2$,
$\upsilon$ is a~$\mathbb{P}^1$-bundle, and $\nu$ is a~conic bundle.
Let $F$ be the~$\varphi$-exceptional~surface.
Then~$\beta(F)=2-S_X(F)$.

Let $\widetilde{H}=\varphi^*(H)$, let $\widetilde{E}_1$ and $\widetilde{E}_2$ be the~proper transforms on $\widetilde{X}$
of the~$\pi$-exceptional surfaces $E_1$ and $E_2$, respectively,
let $\widetilde{S}_1$ and $\widetilde{S}_2$ be the~proper transforms of the~quadric cones in $Q$ that contain $C_1$ and $C_2$, respectively.
We have $\varpi^*(-K_{X})\sim 3\widetilde{H}-\widetilde{E}_1-\widetilde{E}_2$.
We also have $\widetilde{S}_1\sim \widetilde{H}-\widetilde{E}_1-F$ and $\widetilde{S}_2\sim \widetilde{H}-\widetilde{E}_2-F$.
Take $u\in\mathbb{R}_{\geqslant 0}$. Then
\begin{equation}
\label{equation:ODP-F}
\varpi^*(-K_{X})-uF\sim_{\mathbb{R}} \frac{3}{2}\big(\widetilde{S}_1+\widetilde{S}_2\big)+\frac{1}{2}\big(\widetilde{E}_1+\widetilde{E}_2\big)+(3-u)F.
\end{equation}
Hence, intersecting $\varpi^*(-K_{X})-uF$ with a~sufficiently general fiber of the~morphism~$\nu$,
we see that  $\varpi^*(-K_{X})-uF$ is pseudoeffective $\iff$ $u\in[0,3]$. Moreover, we have
\begin{align*}
\widetilde{H}^3&=2, &  \widetilde{E}_1^3&=-4, & \widetilde{E}_2^3&=-4, &  F^3&=2, & \widetilde{E}_1\cdot \widetilde{H}^2&=0,\\
\widetilde{E}_2\cdot \widetilde{H}^2&=0, & \widetilde{E}_2\cdot \widetilde{E}_1^2&=0, & \widetilde{E}_1\cdot \widetilde{E}_2^2&=0, & F\cdot \widetilde{E}_1^2&=0, & F\cdot \widetilde{E}_2^2&=0,\\
\widetilde{E}_1\cdot F^2&=0, & \widetilde{E}_2\cdot F^2&=0, & \widetilde{H}\cdot F^2&=0, & F\cdot \widetilde{H}^2&=0, & \widetilde{H}\cdot \widetilde{E}_1^2&=-2,\\
\widetilde{H}\cdot \widetilde{E}_2^2&=-2, & \widetilde{E}_1\cdot \widetilde{E}_2\cdot \widetilde{H}&=0, & \widetilde{E}_1\cdot \widetilde{E}_2\cdot F&=0, & \widetilde{E}_1\cdot F\cdot \widetilde{H}&=0, & \widetilde{E}_2\cdot F\cdot \widetilde{H}&=0.
\end{align*}
Furthermore, the~divisor $\varpi^*(-K_{X})-uF$ is nef for $u\in[0,1]$. Thus, we have
\begin{multline*}
S_X(F)=\frac{1}{26}\int\limits_{0}^{1}\big(3\widetilde{H}-\widetilde{E}_1-\widetilde{E}_2-uF\big)^3du+\frac{1}{26}\int\limits_{1}^{3}\mathrm{vol}\big(\varpi^*(-K_{X})-uF\big)du=\\
=\frac{1}{26}\int\limits_{0}^{1}26-2u^3du+\frac{1}{26}\int\limits_{1}^{3}\mathrm{vol}\big(\varpi^*(-K_{X})-uF\big)du=\frac{51}{52}+\frac{1}{26}\int\limits_{1}^{3}\mathrm{vol}\big(\varpi^*(-K_{X})-uF\big)du.
\end{multline*}
If $\ell$ is the~preimage on $\widetilde{X}$ of a~general ruling of one of the~cones $\pi\circ\varphi(\widetilde{S}_1)$~and~$\pi\circ\varphi(\widetilde{S}_2)$,
then $(\varpi^*(-K_{X})-uF)\cdot \ell=2-u$, which shows that $\varpi^*(-K_{X})-uF$ is not nef for $u>2$.

Let $\widetilde{L}_1$ and $\widetilde{L}_2$ be the~proper transforms on $\widetilde{X}$ of the curves $L_1$ and $L_2$, respectively.
Then $\varpi^*(-K_{X})-uF$ is not nef for $u>1$, because
$$
\big(\varpi^*(-K_{X})-uF\big)\cdot \widetilde{L}_1=\big(\varpi^*(-K_{X})-uF\big)\cdot \widetilde{L}_2=1-u.
$$
Using \eqref{equation:ODP-F}, one can show that $\widetilde{L}_1$ and $\widetilde{L}_2$
are the~only curves in $\widetilde{X}$ that have negative intersection with $\varpi^*(-K_{X})-uF$ for $u\in(1,2]$.

Note that $\widetilde{L}_1\cong\widetilde{L}_2\cong\mathbb{P}^1$ and their normal bundles  are isomorphic to $\mathcal{O}_{\mathbb{P}^1}(-1)\oplus\mathcal{O}_{\mathbb{P}^1}(-1)$.
Let $\psi\colon\widehat{X}\to\widetilde{X}$ be the~blow up of $\widetilde{L}_1$ and $\widetilde{L}_2$,
let $G_1$ and $G_2$ be the~$\psi$-exceptional surfaces such that $\psi(G_1)=\widetilde{L}_1$ and $\psi(G_2)=\widetilde{L}_2$,
let $\widehat{H}$, $\widehat{E}_1$, $\widehat{E}_2$, $\widehat{S}_1$, $\widehat{S}_2$, $\widehat{F}$,
be the~proper transforms on the~threefold $\widehat{X}$ of the~surfaces  $\widetilde{H}$, $\widetilde{E}_1$, $\widetilde{E}_2$, $\widetilde{S}_1$, $\widetilde{S}_2$, $F$, respectively.
Then
$$
(\varpi\circ\psi)^*(-K_{X})\sim 3\widehat{H}-\widehat{E}_1-\widehat{E}_2.
$$
We have $\widehat{S}_1\sim \widehat{H}-\widehat{E}_1-\widehat{F}-G_1$ and $\widehat{S}_2\sim \widehat{H}-\widehat{E}_2-\widehat{F}-G_2$.
This gives
\begin{equation}
\label{equation:ODP-F-blow-up}
(\varpi\circ\psi)^*(-K_{X})-u\widehat{F}
\sim_{\mathbb{R}}
\frac{3}{2}\big(\widehat{S}_1+\widehat{S}_2\big)+\frac{1}{2}\big(\widehat{E}_1+\widehat{E}_2\big)+\frac{3}{2}\big(G_1+G_2\big)+(3-u)\widehat{F}.
\end{equation}
For  $u\in[0,3]$, the~Zariski decomposition of this divisor can be computed on $\widehat{X}$ as follows:
\begin{itemize}
\item if $u\in[0,1]$, then the~divisor $(\varpi\circ\psi)^*(-K_{X})-u\widehat{F}$ is nef,
\item if $u\in[1,2]$, then the~positive (nef) part of the~Zariski decomposition is
\begin{multline*}
(\varpi\circ\psi)^*(-K_{X})-u\widehat{F}-(u-1)(G_1+G_2)\sim_{\mathbb{R}}\\
\sim_{\mathbb{R}}3\widehat{H}-\widehat{E}_1-\widehat{E}_2-(u-1)(G_1+G_2)\sim_{\mathbb{R}}\\
\sim_{\mathbb{R}}\frac{3}{2}\big(\widehat{S}_1+\widehat{S}_2\big)+\frac{1}{2}\big(\widehat{E}_1+\widehat{E}_2\big)+\frac{5-2u}{2}\big(G_1+G_2\big)+(3-u)\widehat{F},
\end{multline*}
and the~negative part of the~Zariski decomposition is $(u-1)(G_1+G_2)$,
\item if $u\in(2,3]$, then the~positive part of the~Zariski decomposition is
\begin{multline*}
(\varpi\circ\psi)^*(-K_{X})-u\widehat{F}-(u-1)(G_1+G_2)-(u-2)(\widehat{S}_1+\widehat{S}_2)\sim_{\mathbb{R}}\\
\sim_{\mathbb{R}}3\widehat{H}-\widehat{E}_1-\widehat{E}_2-(u-1)(G_1+G_2)-(u-2)(\widehat{S}_1+\widehat{S}_2)\sim_{\mathbb{R}}\\
\sim_{\mathbb{R}}(7-2u)\widehat{H}-(3-u)\big(\widehat{E}_1+\widehat{E}_2+G_1+G_2\big)+(u-4)F\sim_{\mathbb{R}}\\
\sim_{\mathbb{R}}\frac{7-2u}{2}\big(\widehat{S}_1+\widehat{S}_2\big)+\frac{1}{2}\big(\widehat{E}_1+\widehat{E}_2\big)+\frac{5-2u}{2}\big(G_1+G_2\big)+(3-u)\widehat{F},
\end{multline*}
and the~negative part is $(u-1)(G_1+G_2)+(u-2)(\widehat{S}_1+\widehat{S}_2)$.
\end{itemize}
The intersections of the~divisors $\widehat{H}$, $\widehat{E}_1$, $\widehat{E}_2$, $\widehat{F}$, $G_1$, $G_2$ on $\widehat{X}$
can be computed as follows:
\begin{align*}
\widehat{H}^3&=2, & \widehat{E}_1^3&=-4, & \widehat{E}_2^3&=-4, & G_1^3&=2,\\
G_2^3&=2, & \widehat{F}^3&=2, & \widehat{H}\cdot \widehat{E}_1^2&=-2, & \widehat{H}\cdot \widehat{E}_2^2&=-2,\\
\widehat{F}\cdot G_1^2&=-1, & \widehat{F}\cdot G_2^2&=-1, & \widehat{E}_1\cdot G_1^2&=-1, & \widehat{E}_2\cdot G_1^2&=-1,\\
\widehat{E}_1\cdot G_2^2&=-1, & \widehat{E}_2\cdot G_2^2&=-1, & \widehat{H}\cdot G_1^2&=-1, & \widehat{H}\cdot G_2^2&=-1,
\end{align*}
and other triple intersections are zero. Now we are ready to compute $S_X(F)$. We have
\begin{multline*}
S_X(F)=\frac{1}{26}\int\limits_{0}^{1}\Big((\pi\circ\varpi)^*(-K_{X})-u\widehat{F}\Big)^3du+\frac{1}{26}\int\limits_{1}^{2}\Big(3\widehat{H}-\widehat{E}_1-\widehat{E}_2-(u-1)(G_1+G_2)\Big)^3du+\\
+\frac{1}{26}\int\limits_{2}^{3}\Big((7-2u)\widehat{H}-(3-u)\big(\widehat{E}_1+\widehat{E}_2+G_1+G_2\big)+(u-4)F\Big)^3du=\\
=\frac{1}{26}\int\limits_{0}^{1}26-2u^3du+\frac{1}{26}\int\limits_{1}^{2}24+6u-6u^2du+\frac{1}{26}\int\limits_{2}^{3}6(2-u)(4-u)du=\frac{99}{52}.
\end{multline*}
Thus, we see that $\beta(F)=\frac{5}{52}>0$.

Now, suppose that $\beta(\mathbf{F})\leqslant 0$.
Let $\widetilde{Z}$ be the center of the~divisor $\mathbf{F}$ on the~threefold $\widetilde{X}$.
Then $\widetilde{Z}$ is a $G$-invariant irreducible proper subvariety of the~surface~$F$,
since $\beta(F)>0$.

Recall from Lemma~\ref{lemma:ODP-G-action} that $G$ acts faithfully on $F$,
and $F$ does not have $G$-fixed points.
Hence, we conclude that $\widetilde{Z}$ is a~curve.
Let $\widehat{Z}$ be the~proper transform of this curve on~$\widehat{X}$.
Then $\widehat{Z}=C_{\widehat{X}}(\mathbf{F})$ and $\widehat{Z}\subset\widehat{F}$.
Let us apply Abban--Zhuang theory \cite{AbbanZhuang} to the~flag $\widehat{Z}\subset\widehat{F}$.

For every $u\in[0,3]$, let $P(u)$ be the~positive (nef) part of the~Zariski decomposition of the~divisor $(\varpi\circ\psi)^*(-K_{X})-u\widehat{F}$ (described above),
and let $N(u)$ be its negative part.
Observe that $\widehat{Z}\not\subset\mathrm{Supp}(N(u))$, because $\widehat{Z}$ is irreducible and $G$-invariant.
Set
$$
t(u)=\mathrm{sup}\Big\{v\in\mathbb{R}_{\geqslant 0}\ \big\vert\ \text{the~divisor  $P(u)\big\vert_{\widehat{F}}-v\widehat{Z}$ is pseudo-effective}\Big\}.
$$
For $v\in[0,t(u)]$, let $P(u,v)$ be the~nef part of the~Zariski decomposition of~$P(u)\vert_{\widehat{F}}-v\widehat{Z}$,
and let $N(u,v)$ be the~negative part of this~Zariski decomposition. Set
$$
S\big(W^{\widehat{F}}_{\bullet,\bullet};\widehat{Z}\big)=\frac{3}{26}\int\limits_0^3\int\limits_0^{t(u)}\big(P(u,v)\big)^2dvdu.
$$
Then it follows from \cite[Theorem 3.3]{AbbanZhuang} and \cite[Corollary 1.109]{Book} that
$$
1\geqslant\frac{A_X(\mathbf{F})}{S_X(\mathbf{F})}\geqslant\min\Bigg\{\frac{A_X(\widehat{F})}{S_X(\widehat{F})},\frac{1}{S\big(W^{\widehat{F}}_{\bullet,\bullet};\widehat{Z}\big)}\Bigg\}=\min\Bigg\{\frac{104}{99},\frac{1}{S\big(W^{\widehat{F}}_{\bullet,\bullet};\widehat{Z}\big)}\Bigg\}.
$$
Therefore, we conclude that $S(W^{\widehat{F}}_{\bullet,\bullet};\widehat{Z})\geqslant 1$.
Let us show that $S(W^{\widehat{F}}_{\bullet,\bullet};\widehat{Z})<1$.

The morphism $\psi$ induces a~$G$-equivariant birational morphism $\sigma\colon \widehat{F}\to F$
that blows up two points $\widetilde{L}_1\cap F$ and $\widetilde{L}_2\cap F$.
By Lemma~\ref{lemma:ODP-G-action}, $\mathrm{rk}\,\mathrm{Pic}^G(F)=1$, where $F\cong\mathbb{P}^1\times\mathbb{P}^1$.
Thus, we see that $\widehat{F}$ is a~smooth del Pezzo surface of degree $6$, $\mathrm{rk}\,\mathrm{Pic}^G(\widehat{F})=2$,
and there exists the~following $G$-equivariant diagram:
$$
\xymatrix{
&\widehat{F}\ar@{->}[dl]_{\sigma}\ar@{->}[dr]^{\xi}&\\%
F &&\mathbb{P}^1}
$$
where $\xi$ is a~$G$-minimal conic bundle with $2$ singular fibers.
Set $\mathbf{g}_1=G_1\vert_{\widehat{F}}$ and $\mathbf{g}_2=G_2\vert_{\widehat{F}}$.
Let $\mathbf{h}$ be a~curve on $F$ of degree $(1,1)$,
and let $\mathbf{f}$ be a~smooth fiber of the~morphism~$\xi$.
Then $\mathbf{f}\sim\sigma^*\big(\mathbf{h}\big)-\mathbf{g}_1-\mathbf{g}_2$. Since $\widehat{Z}$ is $G$-invariant, we have
$\widehat{Z}\sim n\mathbf{f}+m(\mathbf{g}_1+\mathbf{g}_2)$ for some non-negative integers $n$ and $m$.
Since $\widehat{Z}$ is irreducible, we have $n>0$, so that
$$
S\big(W^{\widehat{F}}_{\bullet,\bullet};\widehat{Z}\big)=\frac{3}{26}\int\limits_0^3\int\limits_0^{\infty}\mathrm{vol}\big(P(u)\vert_{\widehat{F}}-v\widehat{Z}\big)dvdu\leqslant\frac{3}{26}\int\limits_0^3\int\limits_0^{\infty}\mathrm{vol}\big(P(u)\vert_{\widehat{F}}-v\mathbf{f}\big)dvdu.
$$
Hence, to show that $S(W^{\widehat{F}}_{\bullet,\bullet};\widehat{Z})<1$, we may assume that $\widehat{Z}=\mathbf{f}$.

We have $G_1\vert_{\widehat{F}}=\mathbf{g}_1$, $G_2\vert_{\widehat{F}}=\mathbf{g}_2$,
$\widehat{F}\vert_{\widehat{F}}\sim\sigma^*\big(\mathbf{h}\big)\sim\mathbf{f}+\mathbf{g}_1+\mathbf{g}_2$ and $\widehat{H}\vert_{\widehat{F}}\sim \widehat{E}_1\vert_{\widehat{F}}\sim \widehat{E}_2\vert_{\widehat{F}}\sim 0$.
This gives
$$
P(u)\big\vert_{\widehat{F}}-v\mathbf{f}\sim_{\mathbb{R}}\left\{\aligned
&(u-v)\mathbf{f}+u\big(\mathbf{g}_1+\mathbf{g}_2\big)\ \text{if $0\leqslant u\leqslant 1$}, \\
&(u-v)\mathbf{f}+\mathbf{g}_1+\mathbf{g}_2\ \text{if $1\leqslant u\leqslant 2$}, \\
&(4-u-v)\mathbf{f}+\mathbf{g}_1+\mathbf{g}_2\ \text{if $2\leqslant u\leqslant 3$},
\endaligned
\right.
$$
which implies that
$$
t(u)=\left\{\aligned
&u\ \text{if $0\leqslant u\leqslant 2$}, \\
&4-u\ \text{if $2\leqslant u\leqslant 3$}.
\endaligned
\right.
$$
If $u\in[0,1]$, then
$P(u,v)=(u-v)(\mathbf{f}+\mathbf{g}_1+\mathbf{g}_2)$ and $N(u,v)=v(\mathbf{g}_1+\mathbf{g}_2)$ for $v\in[0,u]$,
which gives $(P(u,v))^2=2(u-v)^2$ for every $(u,v)\in[0,1]\times[0,u]$.
If $u\in[1,2]$, then
$$
P(u,v)=\left\{\aligned
&(u-v)\mathbf{f}+\mathbf{g}_1+\mathbf{g}_2\ \text{if $0\leqslant v\leqslant u-1$}, \\
&(u-v)\big(\mathbf{f}+\mathbf{g}_1+\mathbf{g}_2\big)\ \text{if $u-1\leqslant v\leqslant u$},
\endaligned
\right.
$$
and
$$
N(u,v)=\left\{\aligned
&0\ \text{if $0\leqslant v\leqslant u-1$}, \\
&(v-u+1)\big(\mathbf{g}_1+\mathbf{g}_2\big)\ \text{if $u-1\leqslant u\leqslant u$},
\endaligned
\right.
$$
which gives
$$
\big(P(u,v)\big)^2=\left\{\aligned
&4u-4v-2\ \text{if $0\leqslant v\leqslant u-1$}, \\
&2(u-v)^2\ \text{if $u-1\leqslant v\leqslant u$}.
\endaligned
\right.
$$
Finally, if $u\in[2,3]$, then
$$
P(u,v)=\left\{\aligned
&(4-u-v)\mathbf{f}+\mathbf{g}_1+\mathbf{g}_2\ \text{if $0\leqslant v\leqslant 3-u$}, \\
&(4-u-v)\big(\mathbf{f}+\mathbf{g}_1+\mathbf{g}_2\big)\ \text{if $3-u\leqslant v\leqslant 4-u$}.
\endaligned
\right.
$$
and
$$
N(u,v)=\left\{\aligned
&0\ \text{if $0\leqslant v\leqslant 3-u$}, \\
&(v+u-3)\big(\mathbf{g}_1+\mathbf{g}_2\big)\ \text{if $3-u\leqslant u\leqslant 4-u$},
\endaligned
\right.
$$
which gives
$$
\big(P(u,v)\big)^2=\left\{\aligned
&14-4u-4v\ \text{if $0\leqslant v\leqslant 3-u$}, \\
&2(4-u-v)^2\ \text{if $3-u\leqslant v\leqslant 4-u$}.
\endaligned
\right.
$$
Now, integrating, we get
\begin{multline*}
S(W^{\widehat{F}}_{\bullet,\bullet};\mathbf{f})=\frac{3}{26}\int\limits_0^1\int\limits_0^{u}2(u-v)^2dvdu+\frac{3}{26}\int\limits_1^2\int\limits_0^{u-1}4u-4v-2dvdu+\frac{3}{26}\int\limits_1^2\int\limits_{u-1}^u2(u-v)^2dvdu+\\
+\frac{3}{26}\int\limits_2^3\int\limits_0^{3-u}14-4u-4vdvdu+\frac{3}{26}\int\limits_2^3\int\limits_{3-u}^{4-u}2(4-u-v)^2dvdu=\frac{29}{52}<1,
\end{multline*}
which is a~contradiction. Lemma~\ref{lemma:K-stability-isolated-2} is proved.
\end{proof}

Now, using the~$G$-equivariant valuative criterion for K-polystability \cite[Corollary 4.14]{Zhuang},
we see that Theorem~\ref{theorem:K-stability-isolated} follows from Lemmas~\ref{lemma:K-stability-isolated-1} and \ref{lemma:K-stability-isolated-2}.

\subsection{Threefolds with non-isolated singularities.}
\label{section:beth-special}

We continue to use the notation introduced at the beginning of this section.
Suppose, in addition, that $a\pm b\pm 1=ab=0$. Up to a~change of coordinates, we may assume that
$$
X=\big\{u_1x=v_1y,u_2z=v_2t,w^2+(x+z)(y+t)=0\big\}\subset\mathbb{P}^1\times\mathbb{P}^1\times\mathbb{P}^4,
$$
and $Q=\{w^2+(x+z)(y+t)=0\}\subset\mathbb{P}^4$. Then $X$ is singular along the~curve
$$
C=\big\{w=0,x+z=0,y+t=0,u_1z=v_1t,u_2z=v_2t,v_1u_2=u_1v_2\big\},
$$
and $Q$ is a~quadric with $\mathrm{cA}_1$ singularities along the~line $L=\{w=0,x+z=0,y+t=0\}$.
Recall that $\pi\colon X\to Q$ is a~blow up of the~following two smooth conics:
\begin{align*}
{C}_1&=\{w^2+zt=0,x=0,y=0\},\\
{C}_2&=\{w^2+xy=0,z=0,t=0\},
\end{align*}
which both lie in the~smooth locus of the~quadric $Q$. We have $L=\pi(C)$. In this subsection we will  prove the following result.

\begin{theorem}
\label{theorem:K-stability-non-isolated}
The threefold $X$ from class $(\beth)$ described above, which has non-isolated singularities, is K-polystable.
\end{theorem}

Let $\Pi_1$ and $\Pi_2$ be the~planes $\{x+z=0,w=0\}$ and $\{y+t=0,w=0\}$ in $\mathbb{P}^4$, respectively.
Then $\Pi_1$ and $\Pi_2$ are contained in the~quadric $Q$. In fact, we have $\Pi_1+\Pi_2=Q\cap\{w=0\}$.
Now, we let $S_1$ and $S_2$ be the~proper transforms on $X$ of the~planes $\Pi_1$ and $\Pi_2$, respectively.
Then $S_1\cap S_2=X\cap\{w=0\}$, $\Pi_1\cap\Pi_2=\mathrm{Sing}(Q)=L$ and $S_1\cap S_2=\mathrm{Sing}(X)=C$.

Recall that $\tau_3$ is the~involution in $\mathrm{Aut}(X)$ given by
$$
\big([u_1:v_1],[u_2:v_2],[x:y:z:t:w]\big)\mapsto \big([u_1:v_1],[u_2:v_2],[x:y:z:t:-w]\big).
$$
We set $W=X/\tau_3$, and let $\theta\colon X\to W$ be the~quotient map.
Then $\theta$ is ramified in~$S_1+S_2$, and $W$ is the~smooth Fano threefold in the~family \textnumero 3.25,
which can be identified with
$$
\big\{u_1x=v_1y,u_2z=v_2t\big\}\subset\mathbb{P}^1\times\mathbb{P}^1\times\mathbb{P}^3,
$$
where we consider $([u_1:v_1],[u_2:v_2],[x:y:z:t])$ as coordinates on $\mathbb{P}^1\times\mathbb{P}^1\times\mathbb{P}^3$.
Moreover, we have the~following commutative diagram:
\begin{equation}
\label{equation:beth-double-cover}
\xymatrix{
{Q}\ar@{-->}[drr]\ar@{->}[rrrr]^{\vartheta}&&&&\mathbb{P}^3\ar@{-->}[dll]\\
&&\mathbb{P}^1\times\mathbb{P}^1&&\\
{X}\ar@{->}[rru]_{\eta}\ar@{->}[uu]^{\pi}\ar@{->}[rrrr]_{\theta}&&&&{W}\ar@{->}[uu]_{\varpi}\ar@{->}[ull]}
\end{equation}
where
\begin{itemize}
\item  $\vartheta$ is the~double cover given by $[x:y:z:t:w]\mapsto[x:y:z:t]$,

\item $\varpi$ is a blow up of the~lines $\vartheta(C_1)=\{x=0,y=0\}$ and $\vartheta(C_2)=\{z=0,t=0\}$,

\item the~map $Q\dasharrow \mathbb{P}^1\times\mathbb{P}^1$ is given by $[x:y:z:t:w]\mapsto([x:y],[z:t])$,

\item the~map $\mathbb{P}^3\dasharrow\mathbb{P}^1\times\mathbb{P}^1$ is given by $[x:y:z:t]\mapsto([x:y],[z:t])$,

\item $\eta$ and $W\to\mathbb{P}^1\times\mathbb{P}^1$ are natural projections.
\end{itemize}
Observe that the~double cover $\vartheta$ is ramified in $\Pi_1+\Pi_2$,
the morphism $\eta$ is a conic bundle, and the morphism $W\to\mathbb{P}^1\times\mathbb{P}^1$ is a $\mathbb{P}^1$-bundle.

\begin{remark}
\label{remark:referee}
It follows from \cite{Dervan,Fujita2019b,LiuZhu,Zhuang} that
\begin{center}
$X$ is K-polystable $\iff$ the~log Fano pair $\Big(W,\frac{1}{2}\big(\theta(S_1)+\theta(S_2)\big)\Big)$ is K-polystable.
\end{center}
Moreover, as it was pointed to us by the anonymous referee,
it is not very difficult to prove that $(W,\frac{1}{2}(\theta(S_1)+\theta(S_2))$ is K-polystable using the symmetries of this log Fano~pair.
Nevertheless, we decided to prove the K-polystability of the Fano threefold $X$ directly for the consistency of the exposition.
\end{remark}

In addition to the~finite subgroup $G\cong\mumu_2^3$ described in the beginning of this section,
the~group $\mathrm{Aut}(X)$~contains a~subgroup $\Gamma\cong\mathbb{C}^\ast$ that consists of the~automorphisms
$$
\big([u_1:v_1],[u_2:v_2],[x:y:z:t:w]\big)\mapsto \Big(\Big[\frac{u_1}{\lambda}:\lambda v_1\Big],\Big[\frac{u_2}{\lambda}:\lambda v_2\Big],\Big[\lambda x:\frac{y}{\lambda}:\lambda z:\frac{t}{\lambda}:w\Big]\Big),
$$
where $\lambda\in\mathbb{C}^\ast$.
Let $\mathbf{G}$ be the~subgroup in $\mathrm{Aut}(X)$ generated by
the~subgroup $\Gamma$ together with the~involutions $\tau_1$, $\tau_2$, $\tau_3$ described earlier in this section.
Note that $\tau_3\in\Gamma$, so that
$$
\mathbf{G}\cong\big(\mathbb{C}^\ast\rtimes\mumu_2\big)\times\mumu_2,
$$
because $\tau_1\circ\gamma=\gamma\circ\tau_1$ and $\tau_2\circ\gamma\circ\tau_2=\gamma^{-1}$ for every $\gamma\in\Gamma$.
Note also that the~commutative diagram \eqref{equation:beth-double-cover} is $\mathbf{G}$-equivariant,
and the~curve $C$ is $\mathbf{G}$-invariant. Set
$$
C^\prime=\big\{u_1x=v_1y,u_2z=v_2t,w^2+(x+z)(y+t)=0,x=z,t=y\big\}\subset\mathbb{P}^1\times\mathbb{P}^1\times\mathbb{P}^4.
$$
Then $C^\prime$ is a $\mathbf{G}$-invariant irreducible smooth curve such that $\pi(C^\prime)$ is a smooth conic.

\begin{lemma}
\label{lemma:beth-special-curves}
The following assertions holds:
\begin{enumerate}
\item[$(\mathrm{1})$] $X$ does not have $\mathbf{G}$-fixed points;
\item[$(\mathrm{2})$] the~only $\mathbf{G}$-invariant irreducible curves in $X$ are $C$ and $C^\prime$.
\end{enumerate}
\end{lemma}

\begin{proof}
Since  the diagram \eqref{equation:beth-double-cover} is $\mathbf{G}$-equivariant and $\tau_2$ swaps the~$\pi$-exceptional surfaces,
it is enough to prove the~following assertions for the~induced $\mathbf{G}$-action on $\mathbb{P}^3$:
\begin{enumerate}
\item[$(\mathrm{1})$] $\mathbb{P}^3$ does not have $\mathbf{G}$-fixed points;
\item[$(\mathrm{2})$] the~only $\mathbf{G}$-invariant irreducible curves in $\mathbb{P}^3$ are the~lines $\vartheta(L)$ and $\vartheta\circ\pi(C^\prime)$.
\end{enumerate}
Both these assertions are easy to check, since $\tau_1$ and $\tau_2$ act on $\mathbb{P}^3$ as
\begin{align*}
\tau_1\colon [x:y:z:t] &\mapsto [y:x:t:z],\\
\tau_2\colon [x:y:z:t] &\mapsto [z:t:x:y],
\end{align*}
and the~subgroup $\Gamma$ acts on $\mathbb{P}^3$ as $[x:y:z:t:w]\mapsto[\lambda^2 x:y:\lambda^2 z:t]$, where $\lambda\in\mathbb{C}^\ast$.
\end{proof}

Let $Q\dasharrow\mathbb{P}^2$ be the~rational map given by $[x:y:z:t:w]\mapsto[x+z:y+t:w]$.
This~map is undefined along the~line $\pi(C)=\mathrm{Sing}(Q)=\{x+z=0,y+t=0,w=0\}$,
and the~closure of its image is a smooth conic $\mathscr{C}\subset\mathbb{P}^2$.
Therefore, we have a $\mathbf{G}$-equivariant dominant map $\chi\colon Q\dasharrow \mathscr{C}$,
which fits the~following $\mathbf{G}$-equivariant commutative diagram:
\begin{equation}
\label{equation:beth-big}
\xymatrix{
&\mathbb{P}^1\times\mathbb{P}^1&&\\
&\widetilde{X}\ar@{->}[dl]_{\phi}\ar@{->}[dr]\ar@{->}[rr]^{\rho}&&\mathscr{C}\times\mathbb{P}^1\times\mathbb{P}^1\ar@{->}[dd]^{\mathrm{pr}_1}\ar@/_2pc/@{->}[ull]_{\mathrm{pr}_{2,3}}\\%
X\ar@/^1pc/@{->}[uur]^{\eta}\ar@{->}[dr]_{\pi}&&\widetilde{Q}\ar@{->}[dl]\ar@{->}[rd]&&\\%
&{Q}\ar@{-->}[rr]_{\chi}&&\mathscr{C}}
\end{equation}
where both $\widetilde{X}$ and $\widetilde{Q}$ are smooth threefolds, and
\begin{itemize}
\item $\phi$ is the~blow of the~curve $C$,
\item $\widetilde{Q}\to Q$ is the~blow up of the~line~$L$,
\item $\widetilde{X}\to\widetilde{Q}$ is the~blow up of the~preimages of the~conics $C_1$~and~$C_2$,
\item $\mathrm{pr}_{1}$ is the~projection to the~first factor,
\item $\mathrm{pr}_{2,3}$ is the~projection to the~product of the~second and the~third factors,
\item $\widetilde{Q}\to\mathscr{C}$ is a~$\mathbb{P}^2$-bundle,
\item $\rho$ is a birational morphism.
\end{itemize}
To describe $\rho$ more explicitly,
let~$\mathscr{S}$ be the~surface in $Q$ that is cut out by $xt=yz$,
and let $\widetilde{\mathscr{S}}$ be its proper transform on $\widetilde{X}$.
Then $\rho$ contracts the~surface $\widetilde{\mathscr{S}}$ to a smooth curve.
Note that the~birational map $Q\dasharrow\mathscr{C}\times\mathbb{P}^1\times\mathbb{P}^1$ in \eqref{equation:beth-big} is given by
$$
[x:y:z:t:w]\mapsto\big([x+z:y+t:w],[x:y],[z:t]\big).
$$
Furthermore, there exists a $\mathbf{G}$-equivariant ramified double cover $\upsilon\colon\mathscr{C}\to\mathbb{P}^1$
that fits into the~following $\mathbf{G}$-equivariant commutative diagram:
\begin{equation}
\label{equation:beth-quotient-long}
\xymatrix{
Q\ar@{->}[d]_{\vartheta}&&X\ar@{->}[ll]_{\pi}\ar@{->}[d]_{\theta} &&\widetilde{X}\ar@{->}[ll]_{\phi}\ar@{->}[d]_{\tilde{\theta}}\ar@{->}[rr]^(.4){\rho}&&\mathscr{C}\times\mathbb{P}^1\times\mathbb{P}^1\ar@{->}[d]^{\upsilon\times\mathrm{Id}_{\mathbb{P}^1}\times\mathrm{Id}_{\mathbb{P}^1}}\\%
\mathbb{P}^3&& W\ar@{->}[ll]_{\varpi}&& \widetilde{W}\ar@{->}[ll]_{\varphi}\ar@{->}[rr]^(.4){\varrho}&&\mathbb{P}^1\times\mathbb{P}^1\times\mathbb{P}^1}
\end{equation}
where $\tilde{\theta}$ is the~quotient by the~involution $\tau_3$,
the~map $\varphi$ is the~blow up of the~curve $\theta(C)$, and $\varrho$ is the~blow up of the~tridiagonal.
Note that $\widetilde{W}$ is the~unique smooth Fano threefold in the~deformation family \textnumero 4.6,
and $\varpi\circ\varphi$ blows up the~lines $\vartheta(L)$, $\vartheta(C_1)$, $\vartheta(C_2)$.

Let $H$ be the~proper transform on $X$ of a~general hyperplane section of the~quadric $Q$,
let $E_1$ and $E_2$ be the~$\pi$-exceptional surfaces such that $\pi(E_1)=C_1$ and $\pi(E_2)=C_2$. Then
$$
-K_X\sim 3H-E_1-E_2.
$$
Let $\widetilde{H}$, $\widetilde{E}_1$, $\widetilde{E}_2$ be the~proper transforms on $\widetilde{X}$ of the~surfaces $H$, $E_1$, $E_2$,
respectively, and let $F$ be the~$\phi$-exceptional surface.
Then $-K_{\widetilde{X}}\sim\phi^*(-K_X)\sim 3\widetilde{H}-\widetilde{E}_1-\widetilde{E}_2$ and
$$
\widetilde{\mathscr{S}}\sim 2\widetilde{H}-\widetilde{E}_1-\widetilde{E}_2-F.
$$
Note that $F$ and $\widetilde{\mathscr{S}}$ are $\mathbf{G}$-invariant prime divisors in $\widetilde{X}$,
which intersect transversally along a smooth irreducible $\mathbf{G}$-invariant curve.
Furthermore, using \eqref{equation:beth-quotient-long}, we also see that the~surface $\widetilde{\theta}(F)$ is $\varphi$-exceptional,
and the~surface $\widetilde{\theta}(\widetilde{\mathscr{S}})$ is $\varrho$-exceptional.

\begin{lemma}
\label{lemma:beth-G-action}
The following assertions hold:
\begin{enumerate}[(a)]
\item One has $F\cong\mathbb{P}^1\times\mathbb{P}^1$,
and  $F$ contains exactly two $\mathbf{G}$-invariant irreducible curves,
which are both smooth and contained in the~linear system $|\widetilde{\mathscr{S}}\vert_{F}|$.
\item One has $\widetilde{\mathscr{S}}\cong\mathbb{P}^1\times\mathbb{P}^1$,
and $\widetilde{\mathscr{S}}$ contains exactly two $\mathbf{G}$-invariant irreducible curves,
the intersection $F\cap \widetilde{\mathscr{S}}$ and the~proper transform of the~curve $C^\prime$,
which are both smooth and contained in the~linear system $|F\vert_{\widetilde{\mathscr{S}}}|$.
\end{enumerate}
\end{lemma}

\begin{proof}
An easy explicit calculation using \eqref{equation:beth-quotient-long} and the~formulas for the~$\mathbf{G}$-action.
\end{proof}

On $\widetilde{X}$, intersections of the~divisors  $\widetilde{H}$, $\widetilde{E}_1$, $\widetilde{E}_2$, $F$ can be computed as follows:
$$
\widetilde{H}^3=2, \widetilde{E}_1^3=\widetilde{E}_2^3=-4, F^3=-4, \widetilde{H}\cdot F^2=-2, \widetilde{H}\cdot \widetilde{E}_1^2=\widetilde{H}\cdot \widetilde{E}_2^2=-2,
$$
and all other triple intersections are zero.

\begin{lemma}
\label{lemma:beth-special-Eff-cone}
Let $S$ be an $\mathbf{G}$-invariant prime divisor in $\widetilde{X}$. Then $\beta(S)>0$.
\end{lemma}

\begin{proof}
Let $\widetilde{S}_1$ and $\widetilde{S}_2$ be  proper transforms on $\widetilde{X}$ of the~surfaces $S_1$ and $S_2$, respectively.
Using \eqref{equation:beth-big}, one can show that
$S\sim a\widetilde{H}+b\big(\widetilde{E}_1+\widetilde{E}_2\big)+c\big(\widetilde{S}_1+\widetilde{S}_2\big)+d\widetilde{\mathscr{S}}+eF$
for some non-negative integers $a$, $b$, $c$, $d$, $e$.
Thus, to prove the~lemma, it~is~enough to show that
$$
\frac{1}{26}\int\limits_{0}^{\infty}\mathrm{vol}\big(-K_{\widetilde{X}}-uD\big)du<1
$$
for each $D\in\{\widetilde{H}, \widetilde{E}_1+\widetilde{E}_2, \widetilde{S}_1+\widetilde{S}_2, \widetilde{\mathscr{S}}, F\}$.
The computation in the first three cases~is~very similar to the~proof of Corollary~\ref{corollary:divisorial-stability}; we'll illustrate the cases $D=F$ and $D=\widetilde{\mathscr{S}}$.

First, we compute $\beta(F)$. Let $u$ be a~non-negative real number. Then
$$
-K_{\widetilde{X}}-uF\sim_{\mathbb{R}}3\widetilde{H}-\widetilde{E}_1-\widetilde{E}_2-uF\sim_{\mathbb{R}}\widetilde{\mathscr{S}}+\widetilde{S}_1+\widetilde{S}_2+(2-u)F,
$$
because $2\widetilde{S}_1\sim 2\widetilde{S}_2\sim \widetilde{H}-F$.
This shows that $-K_{\widetilde{X}}-uF$ is pseudo-effective $\iff$ $u\leqslant 2$.
Similarly, we see that $-K_{\widetilde{X}}-uF$ is nef $\iff$ $u\in[0,1]$.
Furthermore, if $u\in[1,2]$, then the~Zariski decomposition of this divisor can be described as follows:
$$
-K_{\widetilde{X}}-uF\sim_{\mathbb{R}}\underbrace{-K_{\widetilde{X}}-uF-(u-1)\widetilde{\mathscr{S}}}_{\text{positive part}}+\underbrace{(u-1)\widetilde{\mathscr{S}}}_{\text{negative part}}.
$$
Hence, we have
\begin{multline*}
\beta(F)=1-\frac{1}{26}\int\limits_{0}^{1}\big(-K_{\widetilde{X}}-uF\big)^2du-\frac{1}{26}\int\limits_{1}^{2}\big(-K_{\widetilde{X}}-uF-(u-1)\widetilde{\mathscr{S}}\big)^2du=\\
=1-\frac{1}{26}\int\limits_{0}^{1}\big(3\widetilde{H}-\widetilde{E}_1-\widetilde{E}_2-uF\big)^2du-\frac{1}{26}\int\limits_{1}^{2}\big((5-2u)\widetilde{H}-(2-u)\widetilde{E}_1-(2-u)\widetilde{E}_2-F\big)^2du=\\
=1-\frac{1}{26}\int\limits_{0}^{1}4u^3-18u^2+26du-\frac{1}{26}\int\limits_{1}^{2}12(u-2)^2du=\frac{1}{26}>0.
\end{multline*}

Now, we compute $\beta(\widetilde{\mathscr{S}})$. As above, we observe that
$$
-K_{\widetilde{X}}-u\widetilde{\mathscr{S}}\sim_{\mathbb{R}}
\frac{3-2u}{2}\widetilde{\mathscr{S}}+\frac{1}{2}\big(\widetilde{E}_1+\widetilde{E}_2\big)+\frac{3}{2}F.
$$
This shows that the~divisor $-K_{\widetilde{X}}-u\widetilde{\mathscr{S}}$ is pseudo-effective if and only if $u\leqslant\frac{3}{2}$.
If $u\in[0,1]$, then the~positive part of the~Zariski decomposition of this divisor is
$$
\frac{3-2u}{2}\big(\widetilde{\mathscr{S}}+F\big)+\frac{1}{2}\big(\widetilde{E}_1+\widetilde{E}_2\big),
$$
and the~negative part is $uF$.
If $1\leqslant u\leqslant\frac{3}{2}$, the~Zariski decomposition is the~following:
$$
-K_{\widetilde{X}}-u\widetilde{\mathscr{S}}
\sim_{\mathbb{R}}\underbrace{
\frac{3-2u}{2}\big(\widetilde{\mathscr{S}}+F+\widetilde{E}_1+\widetilde{E}_2\big)}_{\text{positive part}}+\underbrace{uF+(u-1)\big(\widetilde{E}_1+\widetilde{E}_2\big)}_{\text{negative part}}.
$$
Note that $\widetilde{\mathscr{S}}+F+\widetilde{E}_1+\widetilde{E}_2\sim 2\widetilde{H}$.
Integrating, we get
$\frac{1}{26}\int\limits_{0}^{\infty}\mathrm{vol}\big(-K_{\widetilde{X}}-u\widetilde{\mathscr{S}}\big)du=\frac{49}{104}$.
\end{proof}

Now we are ready to prove that $X$ is K-polystable.
Suppose  that $X$ is not K-polystable. By \cite[Corollary 4.14]{Zhuang},
there exists a~$\mathbf{G}$-invariant prime divisor $\mathbf{F}$ over $X$ with $\beta(\mathbf{F})\leqslant 0$.
Let $Z$ be its center on $\widetilde{X}$.
Then $Z$ is a~$\mathbf{G}$-invariant curve by
Lemmas~\ref{lemma:beth-special-curves} and \ref{lemma:beth-special-Eff-cone}.

\begin{lemma}
\label{lemma:beth-Z-not-F}
One has $Z\not\subset F$.
\end{lemma}

\begin{proof}
Let $\mathbf{s}$ be a~section of the~projection $F\to C$ such that $\mathbf{s}^2=0$,
and let $\mathbf{f}$ be its fiber. Then $\widetilde{E}_1\cap F=\varnothing$ and $\widetilde{E}_2\cap F=\varnothing$.
Moreover, we compute $F\vert_{F}\sim -2\mathbf{s}+\mathbf{f}$ and $\widetilde{H}\vert_{F}\sim\mathbf{f}$,
which gives $\widetilde{\mathscr{S}}\vert_{F}\sim 2\mathbf{s}+\mathbf{f}$.

Suppose that $Z\subset F$. Then $Z\sim 2\mathbf{s}+\mathbf{f}$, by Lemma \ref{lemma:beth-G-action}(a).
For every number $u\in[0,2]$, let $P(u)$ be the~positive part of the~Zariski decomposition of $-K_{\widetilde{X}}-uF$
described in the~proof of Lemma~\ref{lemma:beth-special-Eff-cone}, and let $N(u)$ be its negative part.
Take $v\in\mathbb{R}_{\geqslant 0}$. Then
$$
P(u)\big\vert_{F}-vZ=\left\{\aligned
&(2u-2v)\mathbf{s}+(3-u-v)\mathbf{f}\ \text{if $0\leqslant u\leqslant 1$}, \\
&(2-2v)\mathbf{s}+(4-2u-v)\mathbf{f}\ \text{if $1\leqslant u\leqslant 2$},
\endaligned
\right.
$$
and
$$
N(u)\big\vert_{F}=\left\{\aligned
&0\ \text{if $0\leqslant u\leqslant 1$}, \\
&(u-1)\widetilde{\mathscr{S}}\big\vert_{F}\ \text{if $1\leqslant u\leqslant 2$}.
\endaligned
\right.
$$
For $u\in[0,2]$, set
$t(u)=\mathrm{sup}\{v\in\mathbb{R}_{\geqslant 0}\ \vert\ \text{$P(u)\vert_{F}-vZ$ is pseudo-effective}\}$.
Then
$$
t(u)=\left\{\aligned
&u\ \text{if $0\leqslant u\leqslant 1$}, \\
&1\ \text{if $1\leqslant u\leqslant \frac{3}{2}$}, \\
&4-2u\ \text{if $\frac{3}{2}\leqslant u\leqslant 2$}.
\endaligned
\right.
$$
Moreover, it follows from \cite[Theorem 3.3]{AbbanZhuang} and \cite[Corollary 1.109]{Book} that
$$
1\geqslant\frac{A_X(\mathbf{F})}{S_X(\mathbf{F})}\geqslant\min\Bigg\{\frac{1}{S_X(F)},\frac{1}{S\big(W^{F}_{\bullet,\bullet};Z\big)}\Bigg\},
$$
where
$$
S\big(W^{F}_{\bullet,\bullet};Z\big)=
\frac{3}{26}\int\limits_0^2\big(P(u)\big\vert_{F}\big)^2\mathrm{ord}_{Z}\big(N(u)\big\vert_{F}\big)du+\frac{3}{26}\int\limits_0^2\int\limits_0^{t(u)}\big(P(u)\vert_{F}-vZ\big)^2dvdu.
$$
In the~proof of Lemma~\ref{lemma:beth-special-Eff-cone}, we computed $S_X(F)=\frac{25}{26}$, so $S(W^{F}_{\bullet,\bullet};Z)\geqslant 1$.
But
\begin{multline*}
S\big(W^{F}_{\bullet,\bullet};Z\big)=
\frac{3}{26}\int\limits_1^2\big(16-8u\big)^2(u-1)\mathrm{ord}_{Z}\big(\widetilde{\mathscr{S}}\big\vert_{F}\big)du+\\
+\frac{3}{26}\int\limits_0^1\int\limits_0^{u}4(3-u-v)(u-v)dvdu+\frac{3}{26}\int\limits_1^\frac{3}{2}\int\limits_0^{1}4(1-v)(4-2u-v)dvdu+\\
+\frac{3}{26}\int\limits_\frac{3}{2}^2\int\limits_0^{4-2u}4(1-v)(4-2u-v)dvdu=
\frac{2}{13}\mathrm{ord}_{Z}\big(\widetilde{\mathscr{S}}\big\vert_{F}\big)+\frac{33}{104}\leqslant \frac{49}{104}<1,
\end{multline*}
which is a contradiction
\end{proof}

\begin{lemma}
\label{lemma:beth-Z-not-S}
One has $Z\not\subset \widetilde{\mathscr{S}}$.
\end{lemma}

\begin{proof}
Set $\mathbf{s}=F\vert_{\widetilde{\mathscr{S}}}$.
Then $\mathbf{s}$ is a ruling of $\widetilde{\mathscr{S}}\cong\mathbb{P}^1\times\mathbb{P}^1$.
Let $\mathbf{f}$ be a~ruling of this surface such that $\mathbf{s}\cdot \mathbf{f}=1$ and $\mathbf{f}^2=0$.
Then $\widetilde{E}_1\vert_{\widetilde{\mathscr{S}}}\sim \widetilde{E}_2\vert_{\widetilde{\mathscr{S}}}\sim \mathbf{s}$,
$\widetilde{H}\vert_{\widetilde{\mathscr{S}}}\sim\mathbf{s}+2\mathbf{f}$, $\widetilde{\mathscr{S}}\vert_{\widetilde{\mathscr{S}}}\sim -\mathbf{s}+4\mathbf{f}$.

Suppose $Z\subset \widetilde{\mathscr{S}}$. Then $Z\sim \mathbf{s}$ by Lemma~\ref{lemma:beth-G-action}(b).
For $u\in[0,\frac{3}{2}]$, let $P(u)$ be the~positive part of the~Zariski decomposition of $-K_{\widetilde{X}}-u\widetilde{\mathscr{S}}$
described in the~proof of Lemma~\ref{lemma:beth-special-Eff-cone}, and let $N(u)$ be its negative part.
Take $v\in\mathbb{R}_{\geqslant 0}$. Then
$$
P(u)\big\vert_{\widetilde{\mathscr{S}}}-vZ=\left\{\aligned
&(1-v)\mathbf{s}+(6-4u)\mathbf{f} \ \text{if $0\leqslant u\leqslant 1$}, \\
&(3-2u-v)\mathbf{s}+(6-4u)\mathbf{f}\ \text{if $1\leqslant u\leqslant \frac{3}{2}$},
\endaligned
\right.
$$
and
$$
N(u)\big\vert_{\widetilde{\mathscr{S}}}=\left\{\aligned
&u\mathbf{s}\ \text{if $0\leqslant u\leqslant 1$}, \\
&u\mathbf{s}+(u-1)\widetilde{E}_1\big\vert_{\widetilde{\mathscr{S}}}+(u-1)\widetilde{E}_1\big\vert_{\widetilde{\mathscr{S}}}\ \text{if $1\leqslant u\leqslant \frac{3}{2}$}.
\endaligned
\right.
$$
Note that $Z\not\subset \widetilde{E}_1\cup \widetilde{E}_2$, because $Z$ is $\mathbf{G}$-invariant.

For $u\in[0,\frac{3}{2}]$, set
$t(u)=\mathrm{sup}\{v\in\mathbb{R}_{\geqslant 0}\ \vert\ \text{$P(u)\vert_{\widetilde{\mathscr{S}}}-vZ$ is pseudo-effective}\}$.
Then
$$
t(u)=\left\{\aligned
&1\ \text{if $0\leqslant u\leqslant 1$}, \\
&3-2u\ \text{if $1\leqslant u\leqslant \frac{3}{2}$}.
\endaligned
\right.
$$
Moreover, it follows from \cite[Theorem 3.3]{AbbanZhuang} and \cite[Corollary 1.109]{Book} that
$$
1\geqslant\frac{A_X(\mathbf{F})}{S_X(\mathbf{F})}\geqslant\min\Bigg\{\frac{1}{S_X(\widetilde{\mathscr{S}})},\frac{1}{S\big(W^{\widetilde{\mathscr{S}}}_{\bullet,\bullet};Z\big)}\Bigg\},
$$
where
$$
S\big(W^{\widetilde{\mathscr{S}}}_{\bullet,\bullet};Z\big)=
\frac{3}{26}\int\limits_0^\frac{3}{2}\big(P(u)\big\vert_{\widetilde{\mathscr{S}}}\big)^2\mathrm{ord}_{Z}\big(N(u)\big\vert_{\widetilde{\mathscr{S}}}\big)du+\frac{3}{26}\int\limits_0^\frac{3}{2}\int\limits_0^{t(u)}\big(P(u)\vert_{\widetilde{\mathscr{S}}}-vZ\big)^2dvdu.
$$
We know from the~proof of Lemma~\ref{lemma:beth-special-Eff-cone} that $S_X(\widetilde{\mathscr{S}})=\frac{49}{54}$.
Then $S(W^{\widetilde{\mathscr{S}}}_{\bullet,\bullet};Z)\geqslant 1$. But
\begin{multline*}
S\big(W^{\widetilde{\mathscr{S}}}_{\bullet,\bullet};Z\big)=\frac{3}{26}\int\limits_0^14u(3-2u)\mathrm{ord}_{Z}\big(\mathbf{s}\big)du+\frac{3}{26}\int\limits_1^\frac{3}{2}4u(3-2u)^2\mathrm{ord}_{Z}\big(\mathbf{s}\big)du+\\
+\frac{3}{26}\int\limits_0^1\int\limits_0^{1}4(1-v)(3-2u)dvdu+\frac{3}{26}\int\limits_1^\frac{3}{2}\int\limits_0^{3-2u}4(3-2u)(3-2u-v)dvdu=\frac{49\mathrm{ord}_{Z}\big(\mathbf{s}\big)+51}{104},
\end{multline*}
which implies that $S(W^{\widetilde{\mathscr{S}}}_{\bullet,\bullet};Z)\leqslant \frac{25}{26}$. This is a contradiction.
\end{proof}

By Lemmas~\ref{lemma:beth-Z-not-F} and \ref{lemma:beth-Z-not-S},
we have $Z\not\subset F\cup\widetilde{\mathscr{S}}$.
But $Z\subset F\cup\widetilde{\mathscr{S}}$ by Lemma~\ref{lemma:beth-special-curves}.
The~obtained contradiction implies that the Fano threefold $X$ is K-polystable, completing the proof of Theorem \ref{theorem:K-stability-non-isolated}

\section{Threefolds in the~class $(\aleph)$}
\label{section:aleph}

Let $X$ be the~following complete intersection in $\mathbb{P}^1\times\mathbb{P}^1\times\mathbb{P}^4$:
$$
\big\{u_1x=v_1y, u_2z=v_2t, w^2+r(x+y)^2+r(z+t)^2=(2s+2)(xt+yz)+(2s-2)(xz+yt)\big\},
$$
where $[r:s]\in\mathbb{P}^1$, and $([u_1:v_1],[u_2:v_2],[x:y:z:t:w])$ are coordinates on $\mathbb{P}^1\times\mathbb{P}^1\times\mathbb{P}^4$.
For every $[r:s]$, the~threefold $X$ is singular at the~following two points
\begin{align*}
\big([1:1],[-1:1],[0:0:-1:1:0]\big),\\
\big([-1:1],[1:1],[-1:1:0:0:0]\big),
\end{align*}
which are both isolated threefold nodes ($\mathrm{cA}_1$) in $X$ if $[r:s]\ne [0:1]$.
Furthermore, if~$[r:s]=[\pm 1:1]$, then $X$ has an additional isolated threefold node ($\mathrm{cA}_1$) at the~point
$$
\big([1:1],[1:1],[1:1:\pm 1:\pm 1:0]\big).
$$
Finally, if~$[r:s]=[0:1]$, then $\mathrm{Sing}(X)$ consists of the~following two curves:
\begin{align*}
\big\{x=0&,y=0,w=0,u_1z+v_1t=0, u_2z-v_2t=0, v_1u_2+u_1v_2=0\big\},\\
\big\{z=0&,t=0,w=0,u_1x-v_1y=0, u_2x+v_2y=0, v_1u_2+u_1v_2=0\big\}.
\end{align*}
This is a~very special case (see Section~\ref{section:aleph-0} for more details).

Let $Q$ be the~quadric $\{w^2+r(x+y)^2+r(z+t)^2=(2s+2)(xt+yz)+(2s-2)(xz+yt)\}\subset\mathbb{P}^4$.
Then $Q$ is singular $\iff$ $[r:s]=[\pm 1:1]$. If $[r:s]=[\pm 1:1]$, then
$$
\mathrm{Sing}(Q)=\big[1:1:\pm 1:\pm 1:0\big].
$$
Let $\pi\colon{X}\to{Q}$ be the~birational morphism that is induced by the~projection ${X}\to\mathbb{P}^4$.
Then $\pi$ is a blow up of the~following two singular conics:
\begin{align*}
{C}_1&=\{x=0,y=0,w^2+r(z+t)^2=0\},\\
{C}_2&=\{z=0,t=0,w^2+r(x+y)^2=0\}.
\end{align*}
These conics are contained in the~smooth locus of the~quadric $Q$,
and they are reduced unless $[r:s]=[0:1]$. If $[r:s]=[0:1]$, both conics are double lines.

As in Section~\ref{section:beth}, we see that $\mathrm{Aut}(X)$ contains the~following commuting involutions:
\begin{align*}
\tau_1\colon \big([u_1:v_1],[u_2:v_2],[x:y:z:t:w]\big) &\mapsto \big([v_1:u_1],[v_2:u_2],[y:x:t:z:w]\big),\\
\tau_2\colon \big([u_1:v_1],[u_2:v_2],[x:y:z:t:w]\big) &\mapsto \big([u_2:v_2],[u_1:v_1],[z:t:x:y:w]\big),\\
\tau_3\colon \big([u_1:v_1],[u_2:v_2],[x:y:z:t:w]\big) &\mapsto \big([u_1:v_1],[u_2:v_2],[x:y:z:t:-w]\big).
\end{align*}
Note that $\pi$ is $\mathrm{Aut}(X)$-equivariant, and these involutions act on $\mathbb{P}^4$ as in Section~\ref{section:GIT-2}:
\begin{align*}
\tau_1\colon [x:y:z:t:w] &\mapsto [y:x:t:z:w],\\
\tau_2\colon [x:y:z:t:w] &\mapsto [z:t:x:y:w],\\
\tau_3\colon [x:y:z:t:w] &\mapsto [x:y:z:t:-w].
\end{align*}
These involutions generate a subgroup in  $\mathrm{Aut}(X)$ isomorphic to $\mumu_2^3$.
But $\mathrm{Aut}(X)$ is larger.
In fact, it is infinite.
Indeed, consider the~monomorphism $\mathbb{C}^\ast\hookrightarrow\mathrm{PGL}_4(\mathbb{C})$ given~by
$$
\lambda\mapsto\begin{pmatrix}
\frac{1+\lambda}{2} & \frac{1-\lambda}{2} & 0 & 0 & 0\\
 \frac{1-\lambda}{2} & \frac{1+\lambda}{2} & 0 & 0 & 0\\
 0 & 0 & \frac{\lambda +1}{2 \lambda} & \frac{\lambda -1}{2 \lambda} & 0\\
 0 & 0 & \frac{\lambda -1}{2 \lambda} & \frac{\lambda +1}{2 \lambda} & 0\\
 0 & 0 & 0 & 0 & 1
\end{pmatrix}.
$$
Note that this $\mathbb{C}^\ast$-action leaves $Q$ invariant. This gives a~monomorphism $\mathbb{C}^\ast\hookrightarrow\mathrm{Aut}(Q)$.
Since this $\mathbb{C}^\ast$-action also leaves both planes $\{x=0,y=0\}$ and $\{z=0,t=0\}$ invariant,
it leaves both conics $C_1$ and $C_2$ invariant, because
\begin{align*}
C_1&=Q\cap\{x=0,y=0\},\\
C_2&=Q\cap\{z=0,t=0\}.
\end{align*}
Thus the~$\mathbb{C}^\ast$-action lifts to $X$. Let $\Gamma\cong\mathbb{C}^\ast$ be the~corresponding subgroup in $\mathrm{Aut}(X)$.

Let $\mathbf{G}$ be the~subgroup in $\mathrm{Aut}(X)$ generated by $\Gamma$ together with $\tau_1$, $\tau_2$, $\tau_3$. Then
$$
\mathbf{G}\cong\big(\mathbb{C}^\ast\rtimes\mumu_2\big)\times\mumu_2,
$$
because $\tau_1\in\Gamma$, and $\tau_3\circ\lambda=\lambda\circ\tau_3$ and $\tau_2\circ\lambda\circ\tau_2=\lambda^{-1}$ for every $\lambda\in\Gamma$.

\begin{lemma}
\label{lemma:aleph-G-gixed-points}
If $[r:s]\ne [\pm 1:1]$, then the~Fano threefold $X$ does not have $\mathbf{G}$-fixed points.
If $[r:s]=[\pm 1:1]$, then the~only $\mathbf{G}$-fixed point in $X$ is the~singular point
$$
\big([1:1],[1:1],[1:1:\pm 1:\pm 1:0]\big).
$$
\end{lemma}

\begin{proof}
To simplify the~$\mathbf{G}$-action, let us introduce new coordinates on $\mathbb{P}^4$ as follows:
$$
\left\{\aligned
&\mathbf{x}=x-y, \\
&\mathbf{y}=x+y-z-t, \\
&\mathbf{z}=z-t, \\
&\mathbf{t}=x+y+z+t, \\
&\mathbf{w}=w.
\endaligned
\right.
$$
In new coordinates, the~defining equation of the~quadric $Q$ simplifies to
$$
2\mathbf{x}\mathbf{z}+\frac{r+s}{2}\mathbf{y}^2+\frac{r-s}{2}\mathbf{t}^2+\mathbf{w}^2=0.
$$
Moreover, we have ${C}_1=\{\mathbf{x}=0,\mathbf{y}+\mathbf{t}=0\}\cap Q$ and ${C}_2=\{\mathbf{z}=0,\mathbf{y}-\mathbf{t}=0\}\cap Q$.
The~involutions $\tau_2$ and $\tau_3$ act as
\begin{align*}
\tau_2\colon \big[\mathbf{x}:\mathbf{y}:\mathbf{z}:\mathbf{t}:\mathbf{w}\big]&\mapsto\big[\mathbf{z}:-\mathbf{y}:\mathbf{x}:\mathbf{t}:\mathbf{w}\big],\\
\tau_3\colon \big[\mathbf{x}:\mathbf{y}:\mathbf{z}:\mathbf{t}:\mathbf{w}\big]&\mapsto\big[\mathbf{x}:\mathbf{y}:\mathbf{z}:\mathbf{t}:-\mathbf{w}\big],
\end{align*}
and the~action of the~group $\Gamma\cong\mathbb{C}^\ast$ simplifies to
$$
[\mathbf{x}:\mathbf{y}:\mathbf{z}:\mathbf{t}:\mathbf{w}]\mapsto\Big[\lambda\mathbf{x}:\mathbf{y}:\frac{\mathbf{z}}{\lambda}:\mathbf{t}:\mathbf{w}\Big],
$$
where $\lambda\in\mathbb{C}^\ast$. In new coordinates, the~$\mathbf{G}$-fixed points in $\mathbb{P}^4$ are
\begin{align*}
[0:1:0:0:0],\\
[0:0:0:1:0],\\
[0:0:0:0:1].
\end{align*}
Note that $[0:1:0:0:0]\in Q\iff r+s=0$, and $[0:0:0:1:0]\in Q\iff r-s=0$.
Moreover, we have $[0:0:0:0:1]\not\in Q$ for every $[r:s]\in\mathbb{P}^1$, so the~assertion follows.
\end{proof}

Arguing as in the~proof of Lemma~\ref{lemma:aleph-G-gixed-points}, we obtain the~following result.

\begin{lemma}
\label{lemma:ODP-G-action-aleph}
Suppose that $[r:s]=[\pm 1:1]$. Let $O=([1:1],[1:1],[1:1:\pm 1:\pm 1:0])$,
let $\varphi\colon\widetilde{X}\to X$ be the~blow up of the~point $O$, and let $F$ be the~$\varphi$-exceptional~surface.
Then
\begin{itemize}
\item[$(1)$] $\varphi$ is $\mathbf{G}$-equivariant,
\item[$(2)$] the~group $\mathbf{G}$ acts faithfully on $F$,
\item[$(3)$] $\mathrm{rk}\,\mathrm{Pic}^{\mathbf{G}}(F)=1$,
\item[$(4)$] $F$ does not contain $\mathbf{G}$-fixed points,
\item[$(5)$] $F\cong\mathbb{P}^1\times\mathbb{P}^1$ contains exactly two $\mathbf{G}$-invariant irreducible curves, \\which have bidegree $(1,1)$.
\end{itemize}
\end{lemma}

\begin{proof}
Let us only consider the~case when $[r:s]=[1:1]$; the other case is analogous.
We~use the setup and notation introduced in the~proof of Lemma~\ref{lemma:aleph-G-gixed-points}.
It follows from this proof that  we can $\mathbf{G}$-equivariantly identify the surface $F$ with the quadric surface
$$
\big\{\mathbf{w}^2+2\mathbf{x}\mathbf{z}+\mathbf{y}^2=0\big\}\subset\mathbb{P}^3,
$$
where we consider $[\mathbf{x}:\mathbf{y}:\mathbf{z}:\mathbf{w}]$ as coordinates on $\mathbb{P}^3$.
Now our claims are easy to check.

For instance, if $Z$ is a $\mathbf{G}$-invariant irreducible curve in $F$,
then $\Gamma$ acts faithfully on $Z$, which implies that $\tau_2$ fixes a point on $Z$.
Then, analyzing $\langle\tau_2\rangle$-fixed points in $F$,
we see that either $Z=F\cap\{\mathbf{y}=0\}$ or $Z=F\cap\{\mathbf{w}=0\}$ as claimed.
\end{proof}

We have the~following $G$-equivariant diagram:
$$
\xymatrix{
&X\ar@{->}[dl]_{\pi}\ar@{->}[dr]^{\eta}&\\%
{Q}&&\mathbb{P}^1\times\mathbb{P}^1}
$$
where $\eta$ is the~conic bundle that is given by the~natural projection $\mathbb{P}^1\times\mathbb{P}^1\times\mathbb{P}^4\to \mathbb{P}^1\times\mathbb{P}^1$.
Let $\Delta$ be its discriminant~curve in $\mathbb{P}^1\times\mathbb{P}^1$.
Then $\Delta=\Delta_1+\Delta_2$, where
\begin{align*}
\Delta_1&=\big\{(r-s+1)u_1u_2+(r-s-1)u_1v_2+(r-s-1)u_2v_1+(r-s+1)v_1v_2=0\big\},\\
\Delta_2&=\big\{(r+s-1)u_1u_2+(r+s+1)u_1v_2+(r+s+1)u_2v_1+(r+s-1)v_1v_2=0\big\}.
\end{align*}
If $[r:s]\ne[0:1]$, then they meet transversally at  $([1:-1],[1:1])$ and $([1:1],[1:-1])$.
The curves $\Delta_1$ and $\Delta_2$ are smooth for $[r:s]\ne [\pm 1:1]$.
If $[r:s]=[1:1]$, then
\begin{align*}
\Delta_1&=\big\{(u_1-v_1)(u_2-v_2)=0\big\},\\
\Delta_2&=\big\{u_1u_2+3u_1v_2+3u_2v_1+v_1v_2=0\big\}.
\end{align*}
Similarly, if $[r:s]=[1:-1]$, then
\begin{align*}
\Delta_1&=\big\{3u_1u_2+u_1v_2+u_2v_1+3v_1v_2=0\big\},\\
\Delta_2&=\big\{(u_1-v_1)(u_2-v_2)=0\big\}.
\end{align*}
Finally, if $[r:s]=[0:1]$, then $\Delta_1=\Delta_2$. We will deal with this case in Section~\ref{section:aleph-0}.

\begin{remark}
\label{remark:aleph-P1-P1-fixed-points}
The~only $\mathbf{G}$-fixed points in $\mathbb{P}^1\times\mathbb{P}^1$ are $([1:1],[1:1])$ and $([1:-1],[1:-1])$.
\end{remark}

Suppose $[r:s]\ne[0:1]$.
Let $E_1$ and $E_2$ be the~$\langle\tau_3\rangle$-irreducible reducible $\pi$-exceptional surfaces such that $\pi(E_1)=C_1$ and $\pi(E_2)=C_2$.
Set $H=\pi^*(\mathcal{O}_Q(1))$.
Then
$$
\mathrm{Pic}^{\mathbf{G}}(X)=\mathrm{Cl}^{\mathbf{G}}(X)=\mathbb{Z}[H]\oplus\mathbb{Z}[E_1+E_2].
$$
This easily follows from Lemma~\ref{lemma:ODP-G-action-aleph}, since  $C_1+C_2$ is $\mathbf{G}$-irreducible.

\begin{lemma}
\label{lemma:divisorial-stability-aleph}
Let $S$ be any~$\mathbf{G}$-invariant irreducible surface in $X$. Then $\beta(S)>0$.
\end{lemma}

\begin{proof}
The~conic bundle $\eta\colon X\to\mathbb{P}^1\times\mathbb{P}^1$ is given by the~linear system $|2H+E_1+E_2|$.
Therefore, arguing as in the~proof of Lemma~\ref{lemma:Eff-cone}, we see that
$$
S\sim a(E_1+E_2)+b(2H-E_1-E_2)+cH
$$
for some non-negative integers $a$, $b$, $c$.
So, arguing exactly as in the~proof of Corollary~\ref{corollary:divisorial-stability}, we obtain the~required assertion.
\end{proof}

We conclude this section with the~following technical lemma:

\begin{proposition}
\label{proposition:delta-dP4-aleph}
Let $S$ be a~smooth surface in $|H-E_1|$ or $|H-E_2|$. Then
$$
\delta_P(X)\geqslant\frac{104}{99}
$$
for every point $P\in S$ such that $P\not\in E_1\cup E_2$.
\end{proposition}

\begin{proof}
The proof is identical to the~proof of Proposition~\ref{proposition:delta-dP4}.
\end{proof}

\subsection{Threefolds with two singular points.}
\label{section:aleph-two-points}
As before, we continue to use the notation and \mbox{assumptions} introduced earlier in this section.
Suppose that
$$
[r:s]\ne [\pm 1:1]
$$
and $[r:s]\ne [0:1]$. The goal of this subsection is to prove the following.

\begin{theorem}
\label{theorem:K-stability-aleph-two-singularities}
Let $X$ be a singular Fano threefold from class $(\aleph)$ with two singular points. Then $X$ is K-polystable.
\end{theorem}

The proof of this theorem follows the method used to prove Theorem~\ref{theorem:K-stability-smooth}. We begin~with a technical lemma.

\begin{lemma}
\label{lemma:G-invariant-center-aleph}
Suppose that $X$ is a Fano threefold from class $(\aleph)$ with isolated singularities.
Let $\mathbf{F}$ be a~$\mathbf{G}$-invariant prime divisor over $X$ with $\beta(\mathbf{F}) \leqslant 0$,
and let $Z$ be its center on~$X$. Suppose that $Z$ is not a $\mathbf{G}$-fixed singular point of the threefold $X$.
Then $Z$ is an irreducible curve such that $\eta(Z)=([1:1],[1:1])$ or $\eta(Z)=([1:-1],[1:-1])$.
\end{lemma}

\begin{proof}
The proof is identical to the proof of Lemma~\ref{lemma:G-invariant-center}, using Lemmas~\ref{lemma:aleph-G-gixed-points} and \ref{lemma:divisorial-stability-aleph}, and also Proposition~\ref{proposition:delta-dP4-aleph}.
\end{proof}

Now we prove Theorem \ref{theorem:K-stability-aleph-two-singularities}.
Suppose that the singular threefold $X$ is not K-polystable.
By~\cite[Corollary 4.14]{Zhuang},
there is a~$\mathbf{G}$-invariant prime divisor $\mathbf{F}$ over $X$ such that $\beta(\mathbf{F})\leqslant 0$. Let $Z$ be the~center of  $\mathbf{F}$ on $X$. We seek a contradiction.

By Lemma~\ref{lemma:aleph-G-gixed-points}, the threefold $X$ does not contain $\mathbf{G}$-fixed singular points.
Furthermore, the fibers of the~conic bundle $\eta$ over the~points $([1:1],[1:1])$ and $([1:-1],[1:-1])$ are smooth.
So, applying Lemma~\ref{lemma:G-invariant-center}, we see that  $Z$ is the~fiber of the~conic bundle $\eta$ over one of the points $([1:1],[1:1])$ or $([1:-1],[1:-1])$.

Let $S$ be the~unique surface in $|H-E_1|$ that contains $Z$. Observe the following:
\begin{itemize}
\item[$(\mathrm{1})$] if $\eta(Z)=([1:1],[1:1])$, then $S$ is cut out on $X$ by $u_1-v_1=0$, and
$$
\mathrm{Sing}(S)=\big([1:1],[-1:1],[0:0:-1:1:0]\big)\in\mathrm{Sing}(X),
$$
the surface $\pi(S)$ is a~quadric cone with vertex $[0:0:-1:1:0]=\mathrm{Sing}(C_1)$,
the~intersection $\pi(S)\cap C_2$ consists of two distinct smooth points of the conic $C_2$,
and $\pi$ induces a birational morphism $S\to\pi(S)$ that blows up these two points;

\item[$(\mathrm{2})$] if $\eta(Z)=([1:-1],[1:-1])$, then $S$ is cut out on $X$ by $u_1+v_1=0$, and
$$
\mathrm{Sing}(S)=\big([1:-1],[1,1],[-1:1:0:0:0]\big)\in\mathrm{Sing}(X),
$$
the~quadric surface $\pi(S)$ is smooth, $\mathrm{Sing}(C_2)\in\pi(S)$,
and $\pi$ induces a birational morphism $S\to\pi(S)$,
which is a~weighted blow up of the~point $\pi(S)\cap C_2$.
\end{itemize}
In both cases, the~surface $S$ is a~singular sextic del Pezzo surface with one ordinary double point,
and the~curve $Z$ is contained in the~smooth locus of the~surface $S$.

Let us apply Abban--Zhuang theory \cite{AbbanZhuang,Book} to the~flag $Z\subset S$.
Take $u\in\mathbb{R}_{\geqslant 0}$. Set
$$
P(u)=\left\{\aligned
&(3-u)H-(1-u)E_1-E_2\ \text{if $0\leqslant u\leqslant 1$}, \\
&(3-u)H-E_2\ \text{if $1\leqslant u\leqslant 2$},
\endaligned
\right.
$$
and
$$
N(u)=\left\{\aligned
&0\ \text{if $0\leqslant u\leqslant 1$}, \\
&(u-1)E_1\ \text{if $1\leqslant u\leqslant 2$}.
\endaligned
\right.
$$
Then it follows from the~proof of Proposition~\ref{proposition:delta-dP4}
that $P(u)$ and $N(u)$ are the~positive and the~negative parts of the~Zariski decomposition of $-K_X-uS$ for $u\in[0,2]$,
respectively.
Moreover, if $u>2$, then $-K_X-uS$ is not pseudo-effective. This gives $S_X(S)=\frac{3}{4}$.

Now, we take $v\in\mathbb{R}_{\geqslant 0}$ and consider the~divisor $P(u)\vert_{S}-vZ$. For $u\in[0,2]$, set
$$
t(u)=\mathrm{sup}\Big\{v\in\mathbb{R}_{\geqslant 0}\ \big\vert\ \text{the~divisor  $P(u)\big\vert_{S}-vZ$ is pseudo-effective}\Big\}.
$$
Let $P(u,v)$ be the~positive part of the~Zariski decomposition of $P(u)\vert_{S}-vZ$ for $v\leqslant t(u)$.
Note that $Z\not\subset\mathrm{Supp}(N(u))$ for every $u\in[0,2]$. Thus, it follows from \cite{AbbanZhuang,Book} that
$$
\frac{A_X(\mathbf{F})}{S_X(\mathbf{F})}\geqslant\min\Bigg\{\frac{1}{S_X(S)},\frac{1}{S\big(W^S_{\bullet,\bullet};Z\big)}\Bigg\},
$$
where
$$
S\big(W^S_{\bullet,\bullet};Z\big)=\frac{3}{(-K_X)^3}\int\limits_0^2\int\limits_0^{t(u)}\big(P(u,v)\big)^2dvdu.
$$
Since $S_X(S)=\frac{3}{4}$ and $\beta(\mathbf{F})=A_X(\mathbf{F})-S_X(\mathbf{F})\leqslant 0$,
we conclude that $S(W^S_{\bullet,\bullet};Z)\geqslant 1$.

Let us compute $S(W^S_{\bullet,\bullet};Z)$.
Set $\mathbf{e}=E_2\vert_{S}$. Then $\mathbf{e}^2=-2$ and $Z\cdot \mathbf{e}=2$.
We have
$$
P(u)\big\vert_{S}-vZ=\left\{\aligned
&(2-v)Z+\mathbf{e}\ \text{if $0\leqslant u\leqslant 1$}, \\
&(3-u-v)Z+(2-u)\mathbf{e}\ \text{if $1\leqslant u\leqslant 2$},
\endaligned
\right.
$$
which gives
$$
t(u)=\left\{\aligned
&2\ \text{if $0\leqslant u\leqslant 1$}, \\
&3-u\ \text{if $1\leqslant u\leqslant 2$},
\endaligned
\right.
$$
because $\mathrm{Supp}(\mathbf{e})$ is contractible. Furthermore, if $u\in[0,1]$, then
$$
P(u,v)=\left\{\aligned
&(2-v)Z+\mathbf{e}\ \text{if $0\leqslant v\leqslant 1$}, \\
&(2-v)\big(Z+\mathbf{e}\big)\ \text{if $1\leqslant v\leqslant 2$},
\endaligned
\right.
$$
which gives
$$
\big(P(u,v)\big)^2=\left\{\aligned
&6-4v\ \text{if $0\leqslant v\leqslant 1$}, \\
&2(v-2)^2\ \text{if $1\leqslant v\leqslant 2$}.
\endaligned
\right.
$$
Similarly, if $u\in[1,2]$, then
$$
P(u,v)=\left\{\aligned
&(3-u-v)Z+(2-u)\mathbf{e}\ \text{if $0\leqslant v\leqslant 1$}, \\
&(3-u-v)\big(Z+\mathbf{e}\big)\ \text{if $1\leqslant v\leqslant 3-u$},
\endaligned
\right.
$$
which gives
$$
\big(P(u,v)\big)^2=\left\{\aligned
&2(2-u)(4-u-2v)\ \text{if $0\leqslant v\leqslant 1$}, \\
&2(3-u-v)^2\ \text{if $1\leqslant v\leqslant 3-u$}.
\endaligned
\right.
$$
Integrating, we obtain $S(W^S_{\bullet,\bullet};Z)=\frac{3}{4}$. But we already proved earlier that $S(W^S_{\bullet,\bullet};Z)\geqslant 1$.
This is a contradiction. Hence, we see that $X$ is K-polystable.

\subsection{Threefolds with three singular points.}
\label{section:aleph-node}

We continue to use the notation and assumptions introduced at the~beginning of this section.
Suppose that $[r:s]=[\pm 1:1]$.
The goal of this subsection is to prove the following.

\begin{theorem}
\label{theorem:K-stability-aleph-three-singularities}
Let $X$ be a singular threefold from class $(\aleph)$ with three singular~points. Then $X$ is K-polystable.
\end{theorem}

The proof follows the method used to prove Theorem~\ref{theorem:K-stability-isolated}. After changing coordinates, we may assume that $[r:s]=[1:1]$. Then
$$
X=\big\{u_1x=v_1y,u_2z=v_2t,w^2+(x+y)^2+(z+t)^2=4(xt+yz)\big\}\subset\mathbb{P}^1\times\mathbb{P}^1\times\mathbb{P}^4,
$$
and the~singular locus of the~threefold $X$ consists of the~following three points:
\begin{align*}
\big([1:1],[-1:1]&,[0:0:-1:1:0]\big),\\
\big([-1:1],[1:1]&,[-1:1:0:0:0]\big),\\
\big([1:1],[1:1]&,[1:1:1:1:0]\big),
\end{align*}
which are isolated threefold nodes ($\mathrm{cA}_1$) in $X$. We have
$$
Q=\big\{w^2+(x+y)^2+(z+t)^2=4(xt+yz)\big\}\subset\mathbb{P}^1\times\mathbb{P}^4,
$$
the surface $Q$ is a~cone with vertex $[1:1:1:1:0]$,
and $\pi$ is a~blow up of the~conics
\begin{align*}
{C}_1&=\{x=0,y=0,w^2+(z+t)^2=0\},\\
{C}_2&=\{z=0,t=0,w^2+(x+y)^2=0\}.
\end{align*}
Note that $C_1$ and $C_2$ are reduced singular conics that do not contain $\mathrm{Sing}(Q)$.

The fiber of the~conic bundle $\eta$ over the~point $([1:1],[1:1])$ is a~reducible conic $L_1+L_2$,
where $L_1$ and $L_2$ are smooth irreducible curves such that
\begin{align*}
\pi(L_1)&=\{x=y,z=t,w+2i(x-z)=0\},\\
\pi(L_2)&=\{x=y,z=t,w-2i(x-z)=0\}.
\end{align*}
Note that $\pi(L_1)$ and $\pi(L_2)$ are the~only lines in $Q$ that intersect both conics $C_1$ and $C_2$,
and each of them intersects exactly one irreducible component of the~conics $C_1$ and~$C_2$.
Observe that the intersection $L_1\cap L_2$ is the~singular point $([1:1],[1:1],[1:1:1:1:0])$. We denote this singular point by $O$.

\begin{lemma}
\label{lemma:K-stability-3-points-1}
Let $\mathbf{F}$ be a~$\mathbf{G}$-invariant prime divisor $\mathbf{F}$ over the~threefold $X$ with $\beta(\mathbf{F})\leqslant 0$,
and let $Z$ be its center on the~threefold $X$. Then~$Z=O$.
\end{lemma}

\begin{proof}
By Lemma~\ref{lemma:aleph-G-gixed-points}, $O$ is the~unique $\mathbf{G}$-fixed singular point in $X$.
Thus, by Lemma~\ref{lemma:G-invariant-center-aleph}, we see that one of the following three cases holds:
\begin{itemize}
\item[$(\mathrm{1})$] either $Z=O$,
\item[$(\mathrm{2})$] or $Z$ is an irreducible component of the~fiber of $\eta$ over $([1:1],[1:1])$,
\item[$(\mathrm{3})$] or $Z$ is an irreducible component of the~fiber of $\eta$ over $([1:-1],[1:-1])$.
\end{itemize}
In the second case, either $Z=L_1$ or $Z=L_2$, which is impossible, because neither of these curves is $\mathbf{G}$-invariant.
Moreover, the~fiber of $\eta$ over  $([1:-1],[1:-1])$ is irreducible and smooth.
So, either $Z=O$ or $Z$ is the~fiber of the~conic bundle $\eta$ over $([1:-1],[1:-1])$.

Suppose that $Z$ is the~fiber of $\eta$ over $([1:-1],[1:-1])$.
Set $S=X\cap\{u_1+v_1=0\}$.
Then $S$ is a~singular sextic del Pezzo surface that contains~$Z$. Moreover, we have
$$
\mathrm{Sing}(S)=\big([1:-1],[1,1],[-1:1:0:0:0]\big)\in\mathrm{Sing}(X),
$$
this point is an ordinary double point of the~surface $S$,
the~quadric surface $\pi(S)$ is smooth,
and $\pi$ induces a weighted blow up $S\to\pi(S)$ of the~single intersection point $\pi(S)\cap C_2$.
Now, arguing as in the~end of Section~\ref{section:aleph-two-points}, we obtain a~contradiction.
\end{proof}

Now we are ready to prove that $X$ is K-polystable.
Suppose for a contradiction that the~singular Fano threefold $X$ is not K-polystable.
Then there is a~$\mathbf{G}$-invariant prime divisor $\mathbf{F}$ over $X$ such that $\beta(\mathbf{F})\leqslant 0$.
By Lemma~\ref{lemma:K-stability-3-points-1}, its center on $X$ is the~point $O$.

Let $\varphi\colon\widetilde{X}\to X$ be the~blow up of the~point $O$, and let $F$ be the~$\varphi$-exceptional divisor,
and let $\widetilde{Z}$ be the~center on $\widetilde{X}$ of the~prime divisor  $\mathbf{F}$.
Then $\varphi$ is $\mathbf{G}$-equivariant, and
\begin{itemize}
\item[$(\mathrm{1})$]  either $\widetilde{Z}=F$ and $\beta(F)=\beta(\mathbf{F})$, or
\item[$(\mathrm{2})$]  $\widetilde{Z}$ is a~$\mathbf{G}$-invariant irreducible curve in $F$, or
\item[$(\mathrm{3})$] $\widetilde{Z}$ is a~$\mathbf{G}$-fixed point in $F$.
\end{itemize}
Case $(\mathrm{3})$ is impossible by Lemma~\ref{lemma:ODP-G-action-aleph}.
Moreover, in case $(\mathrm{2})$ we have  two choices for $Z$,
as described in the proof of Lemma~\ref{lemma:ODP-G-action-aleph}.
Now, arguing exactly as in the~proof of Lemma~\ref{lemma:K-stability-isolated-2},
we obtain a contradiction. This proves that $X$ is K-polystable.

\subsection{Threefolds with non-isolated singularities.}
\label{section:aleph-0}
It remains to study the~singular threefold $X$ from class~$(\aleph)$ with $[r:s]=[0:1]$. Recall that $X$ is the~complete intersection
$$
\big\{u_1x=v_1y,u_2z=v_2t,w^2=4(xt+yz)\big\}\subset\mathbb{P}^1\times\mathbb{P}^1\times\mathbb{P}^4,
$$
where $([u_1:v_1],[u_2:v_2],[x:y:z:t:w])$ are coordinates on $\mathbb{P}^1\times\mathbb{P}^1\times\mathbb{P}^4$.
The goal of this subsection is to prove the following.

\begin{theorem}
\label{theorem:K-stability-non-isolated-aleph}
The threefold $X$ from class $(\aleph)$ described above, which has non-isolated singularities, is K-polystable.
\end{theorem}

Let
$$
Q=\big\{w^2=4(xt+yz)\big\}\subset\mathbb{P}^4,
$$
and let $\pi\colon{X}\to{Q}$ be the~birational morphism that is induced by the~projection ${X}\to\mathbb{P}^4$.
Then $Q$ is a smooth quadric threefold, and the~morphism $\pi$ is a~weighted blow up with weights $(1,2)$ of the lines ${L}_1=\{x=0,y=0,w=0\}$ and ${L}_2=\{z=0,t=0,w=0\}$.
Note that the~singular locus of ${X}$ consists of two disjoint smooth curves of $\mathrm{cA}_1$ singularities
\begin{align*}
\{v_1u_2+u_1v_2=0,u_1z+v_1t=0,u_2z=v_2t,x=0,y=0,w=0\},\\
\{v_1u_2+u_1v_2=0,u_2x+v_2y=0,u_1x=v_1y,z=0,t=0,w=0\}.
\end{align*}

Let ${S}$ be the~strict transform on ${X}$ of the~surface in ${Q}$ that is cut out on $Q$ by $w=0$,
let~${E}_1$ and  ${E}_2$ be the two $\pi$-exceptional surfaces such that $\pi({E}_1)={L}_1$ and $\pi({E}_2)={L}_2$,
and let $\iota$ be the~involution in $\mathrm{Aut}(X)$ given by
$$
\iota \colon \big([u_1:v_1],[u_2:v_2],[x:y:z:t:w]\big)\mapsto \big([u_1:v_1],[u_2:v_2],[x:y:z:t:-w]\big).
$$
We set $W=X/\iota$ and let $\theta\colon X\to W$ be the~quotient map.
Then $W$ is a~smooth threefold and $\theta$ is ramified in $S\cup E_1\cup E_2$.
Moreover,  we have the~following commutative diagram:
$$
\xymatrix{
{X}\ar@{->}[d]_{\pi}\ar@{->}[rr]^{\theta}&&{W}\ar@{->}[d]^{\varpi}\\%
{Q}\ar@{->}[rr]_{\vartheta}&&\mathbb{P}^3}
$$
where $\vartheta$ is the~double cover ramified in $\pi(S)$ given by $[x:y:z:t:w]\mapsto[x:y:z:t]$
and $\varpi$ is the~blow up of the~disjoint lines: $\{x=y=0\}$ and $\{z=t=0\}$.
Note that $W$ is the~unique smooth Fano threefold in the~family \textnumero 3.25.

The threefold ${X}$ can be also obtained via the~following commutative diagram:
$$
\xymatrix{
{V}\ar@{->}[d]_{\alpha}&&&{Y}\ar@{->}[lll]_{\beta}\\%
{Q}&&&{X}\ar@{->}[lll]_{\pi}\ar@{->}[u]_{\gamma}}
$$
where $\alpha$  blows up  the~lines ${L}_1$ and ${L}_2$,
$\beta$ blows up the~$(-1)$-curves in the~$\alpha$-exceptional surfaces (which are both isomorphic to $\mathbb{F}_1$),
and $\gamma$ contracts the~$\alpha$-exceptional surfaces.
Note that ${V}$ is the~smooth Fano threefold in the~family \textnumero 3.20.
One can check that
$$
\mathrm{Aut}\big({X}\big)\cong\mathrm{Aut}\big({V}\big)\cong\Big(\mathbb{C}^\ast\rtimes\mumu_2\Big)\times\mathrm{PGL}_2\big(\mathbb{C}\big),
$$
and ${S}$ is the~only $\mathrm{Aut}({X})$-invariant prime divisor over ${X}$.

Set ${H}=\pi^*(\mathcal{O}_{Q}(1))$, and observe that $-K_{{X}}\sim 3{H}-2{E}_1-2{E}_2$ and $S\sim {H}-{E}_1-{E}_2$.
Then, arguing as in the~proof of Lemma~\ref{lemma:beth-special-Eff-cone}, we get
\begin{multline*}
\beta({S})=1-\frac{1}{26}\int\limits_{0}^{2}\big(-K_{X}-u{S}\big)^3du-\frac{1}{26}\int\limits_{2}^{3}\big(-K_{X}-u{S}-(u-2)({E}_1+{E}_2)\big)^3du=\\
=1-\frac{1}{26}\int\limits_{0}^{2}\big((3-u){H}+(u-2)({E}_1+{E}_2)\big)^3du-\frac{1}{26}\int\limits_{2}^{3}\big((3-u){H}\big)^3du=\\
=1-\frac{1}{26}\int\limits_{0}^{2}3u^2-18u+26du-\frac{1}{26}\int\limits_{2}^{3}2(3-u)^3du=1-\frac{49}{52}=\frac{3}{52}>0.
\end{multline*}
Thus it follows from \cite[Corollary 4.14]{Zhuang} that ${X}$ is K-polystable.

\section{Proof of Main Theorem}
\label{section:K-moduli}

In Section~\ref{section:GIT-2} we presented the~parameter space $T=\mathbb{P}^4$
whose open part parametrizes all K-polystable smooth Fano threefolds in the~family~\textnumero~3.10,
and described a $(\mathbb{C}^\ast)^2$-action on $T$ such that the GIT quotient
$T/\!/(\mathbb{C}^\ast)^2\cong\mathbb{P}^1\times\mathbb{P}^1$ parametrizes the
Fano threefolds from the Main Theorem, which include all K-polystable smooth Fano threefolds in this  family.
Now, arguing as in \cite[Section~6]{Gokova}, we obtain the assertion of the Main Theorem.

Namely, our construction implies that there is a $\mathbb{Q}$-Gorenstein family of Fano varieties over the~GIT moduli stack $[\mathring{T}/(\mathbb{C}^\ast)^2]$
that contains all K-polystable smooth members of the~family \textnumero 3.10,
where $\mathring{T}$ is the~GIT semistable open subset of the~parameter space $T$.

In Sections~\ref{section:beth} and \ref{section:aleph}, we proved that all GIT-polystable objects in our family are in fact K-polystable Fano threefolds.
So, we obtain a morphism of stacks
$$
\big[\mathring{T}/(\mathbb{C}^\ast)^2\big]\longrightarrow\mathcal{M}^{\mathrm{Kss}}_{3,26},
$$
where $\mathcal{M}^{\mathrm{Kss}}_{3,26}$ is the~K-moduli stack parametrizing K-semistable Fano 3-folds of degree~$26$.
By \cite[Theorem 6.6]{Alper}, this morphism descends to good moduli spaces
$$
\mathring{T}/(\mathbb{C}^\ast)^2\longrightarrow M^{\mathrm{Kps}}_{3,26},
$$
where points of $M^{\mathrm{Kps}}_{3,26}$ parameterize K-polystable Fano threefolds of degree $26$.

\begin{lemma}
\label{lemma:unobstuctedness}
Let $X$ be a K-polystable limit of smooth Fano threefolds in the family \textnumero 3.10, as described in the Main Theorem.
Then the deformations of $X$ are unobstucted.
\end{lemma}

\begin{proof}
We follow the approach from the proof of \cite[Proposition 2.9]{Gokova}.
Since $X$ is a complete intersection in $\mathbb{P}^1 \times \mathbb{P}^1 \times \mathbb{P}^4$ and $X$  has at worst canonical singularities,
it follows from \cite[Theorems 2.3.2 and 2.4.1, Corollary 2.4.2]{Sernesi} and \cite[Propositions 2.4 and 2.6]{Sano} that the~deformations of $X$ are unobstructed if $\mathrm{Ext}^2_{\mathcal{O}_X}(\Omega^1_X,\mathcal{O}_X) = 0$.
As in \cite[\S 1.2]{Sano}, we have
$$
\mathrm{Ext}^2_{\mathcal{O}_X}(\Omega^1_X,\mathcal{O}_X) \cong \mathrm{Ext}^2_{\mathcal{O}_X}(\Omega^1_X \otimes \omega_X,\omega_X) \cong H^1(X,\Omega^1_X \otimes \omega_X)^{\vee}.
$$
Set $B=\{u_1x=v_1y,\, u_2z=v_2t\} \subset \mathbb{P}^1 \times \mathbb{P}^1 \times \mathbb{P}^4$.
Then $B$ is a smooth Fano fourfold, and~$X$ is a divisor in $B$. 
We have $\omega_B \cong \mathcal{O}_B(-1,-1,-3)$ and $\omega_B(X) \cong \mathcal{O}_B(-1,-1,-1)$.

Tensoring the conormal sequence of $X$ in $B$ with $\omega_X$ and taking cohomology, we obtain the exact sequence
$$
H^1\big(X,\Omega^1_B\big\vert_X \otimes \omega_X\big) \longrightarrow H^1\big(X,\Omega^1_X \otimes \omega_X\big) \longrightarrow H^2\big(X,\mathcal{O}_B(-X)\big\vert_X \otimes \omega_X\big).
$$
By adjunction and Kodaira vanishing we have
$$
H^2\big(X,\mathcal{O}_B(-X)\big\vert_X \otimes \omega_X\big) \cong H^2\big(X,\omega_B\big\vert_X\big) = 0.
$$
Moreover, from the exact sequence
$$
0 \longrightarrow \Omega^1_B \otimes \omega_B \longrightarrow \Omega^1_B \otimes \omega_B(X) \longrightarrow \Omega^1_B|_X \otimes \omega_X \longrightarrow 0
$$
we obtain the exact sequence of cohomology groups
$$
H^1\big(B,\Omega^1_B\otimes \omega_B(X)\big) \longrightarrow H^1\big(X,\Omega^1_B\big|_X \otimes \omega_X\big) \longrightarrow H^2\big(B, \Omega^1_B \otimes \omega_B\big).
$$
Then Akizuki-Nakano vanishing theorem gives
$$
H^1\big(B,\Omega^1_B\otimes \omega_B(X)\big) = H^2\big(B, \Omega^1_B \otimes \omega_B\big) = 0.
$$
This gives $H^1(X,\Omega^1_X \otimes \omega_X) = 0$, so $\mathrm{Ext}^2_{\mathcal{O}_X}(\Omega^1_X,\mathcal{O}_X) = 0$, completing the proof.
\end{proof}

From this lemma and Luna's \'{e}tale slice theorem \cite[Theorem 1.2]{AHD}, we see that $M^{\mathrm{Kps}}_{3,26}$ is normal at $[X]$ for all $[X]$ in the image of the morphism $\mathring{T}/(\mathbb{C}^\ast)^2\to M^{\mathrm{Kps}}_{3,26}$. 
Since both good moduli spaces are known to be proper, the~image of this morphism is a connected component of the~K-moduli space. This implies the Main Theorem.

We further note that the morphism $\mathring{T}/(\mathbb{C}^\ast)^2\to M^{\mathrm{Kps}}_{3,26}$ factors as
$$
\xymatrix{
\mathring{T}/(\mathbb{C}^\ast)^2\ar@{->}[rrrr]^(0.55){\text{quotient by an involution}}&&&&\mathbb{P}^2\ar@{->}[rr]^{\varsigma}&&M^{\mathrm{Kps}}_{3,26}},
$$
where  the map $\mathring{T}/(\mathbb{C}^\ast)^2\longrightarrow\mathbb{P}^2$ is described in Section~\ref{section:quotient-space}. 
Note that $\varsigma$ is generically injective and quasi-finite. 
So, by Zariski's Main Theorem, $\varsigma$ is an isomorphism onto 
the~connected component of $M^{\mathrm{Kps}}_{3,26}$ containing K-polystable smoothable Fano threefolds in deformation family \textnumero 3.10.

\end{document}